\documentclass[11pt]{amsart}
\usepackage{amsmath, amsfonts, amsthm, amssymb}
\usepackage[all,cmtip,line]{xy}\usepackage{mathtools}
\usepackage{hyperref} 
\usepackage{array}
\usepackage{enumerate}
\usepackage{srcltx} 

\usepackage{todonotes}
\newtheorem{thm}{Theorem}[section]
\newtheorem{lem}{Lemma}[section]
\newtheorem{prop}[lem]{Proposition}
\newtheorem{cor}[lem]{Corollary}

\newtheorem{defn}[lem]{Definition}
\newtheorem{rem}[lem]{Remark}

\numberwithin{equation}{section}

\newlength{\originalbase}
\title{Poincar\'e type J-equation}

\author{Xiuxiong Chen, Yulun Xu}

\date{\today}

\begin{document}

\begin{abstract}
We introduce a two-parameter continuity path for the J-equation and use it to characterize the solvability of the J-equation for Kähler metrics with Poincaré type singularities along a divisor $D$, allowing simple normal crossings and self-intersections.
On Kähler surfaces, we show that the classical subsolution condition in the smooth setting implies solvability in the Poincaré type setting for any smooth divisor 
$D$. As a consequence, if $X$ contains no curves of negative self-intersections and $K_X[D]$ is ample, then the K-energy is bounded from below on any Poincaré type Kähler class.
In the smooth divisor case, we further analyze the asymptotic behavior of solutions near $D$, and show that existence of a Poincaré type solution implies existence of a solution to the J-equation on $D$.
\end{abstract}

\maketitle

\tableofcontents

\section{Introduction}
\subsection{Background} The central theme in K\"ahler geometry currently is studying Calabi-Donaldson theory on the geometry of extremal K\"ahler metrics.  In 1976, S. T. Yau proved Calabi conjecture \cite{Y} when the first Chern class is either negative or zero (In negative case, T. Aubin has an independent proof in \cite{Au}). In the case of positive first Chern class, there are obstructions to the existence of K\"ahler Einstein (KE) metrics; around 1980s, Yau proposed a conjecture which relates the existence of KE metrics to the stability of underlying tangent bundles. This conjecture was settled in 2012 through a series of work of Chen-Donaldson-Sun \cite{CDS} \cite{CDS2} and \cite{CDS3} and we refer interested readers to this set of papers for further references in the subject of KE metrics.\\

In 1987, S. Bando and T. Mabuchi \cite{BM} proved that, if there is a KE metric in the canonical K\"ahler class, then the K energy functional
has a uniform lower bound. A dacade later (1998), following Donaldson's program in K\"ahler geometry \cite{D2},
Chen \cite{C0} first extended BM's theorem to the more general setting of constant
scalar curvature K\"ahler (cscK) metrics when the first Chern class is negative. It is then fundamentally important to understand when the K energy has a uniform lower bound
in a given K\"ahler class in absence of the existence of cscK metrics. Chen \cite{C2} found a decomposition formula for the K energy functional where he decomposed the K energy
into three terms: the entropy function which always has a lower bound, the $J$ functional and a normalization term. Following Chen's decomposition formula, it is clear that if $J$ functional bounded from below uniformly, so does the K energy functional. In the special case when the first Chern class is negative,  the $J$ functional (sutiably defined) is path convex.
Thus, the existence of a solution to $J$ equation implies that
the K energy has a uniform lower bound, moreover, proper. Decades later, a fundamental theorem  is proved by Chen-Cheng \cite{CC1} \cite{CC2} that, in K\"ahler surface with $C_1 <  0$ and no curves with negative self intersections, then it supports a cscK metrics in every K\"ahler class.\\

In 2011, Auvray initiated the study of extremal K\"ahler metrics with poincar\'e type
singularities and subsequently it has attracted wide attentions (c.f. \cite{A}, \cite{A2}, \cite{A3}, \cite{F2}, \cite{XZ} and \cite{X}).
Inspired by his work, we hope to extend Calabi-Donaldson theory  to the cusp settings, fortunately, there has been a lot in depth discussions on
this "hope" and we refer interested readers to  Székelyhidi \cite{Sz}, Auvray \cite{A3} and xiao \cite{Xi} for more infomation.   As a first step, we introduce the $J$-equation  to this settings
and when we want to attack the existnce problem, we encounters some fundamental new difficulty which is not previously treated in compact case. Following the earlier work of Weinkove-Song \cite{SW}, the subsolution method
has been a key tool in the tradtional method and in the cusp settings,  subsolution (to be as effective as in compact settings) should have certain asymptotic behavior near cusp singularity(Such control seems to be difficult to be preserved under traditional continuity path). This phenomenon is common in complete and noncompact manifolds, for example, in the study of complete Tian-Yau metric \cite{TY1}.

\subsection{ The J-equation}
Let $X$ be a closed K\"aher manifold. Let $\omega_X$ and $\chi$ be two K\"ahler metrics on $X$. Following Chen \cite{C2}, we can define J-equation as:
\begin{equation}\label{e j}
\omega_{\varphi}^n =C_{[\omega],[\chi]} \omega_{\varphi}^{n-1}\wedge \chi,\qquad C_{[\omega],[\chi]}= \frac{\int_{X\setminus D} \omega^n}{\int_{X\setminus D} \omega^{n-1}\wedge \chi}
\end{equation}
Here $\omega_{\varphi}= \omega_X + dd^c \varphi >0$ in the sense of current.  Note that J-equation can be seen as a generalization of the complex Monge-Amp\'ere equation (more generally, complex Hessian equations): 
\begin{equation*}
\omega_{\varphi}^n=e^F \omega_X^n.
\end{equation*}

S. Donaldson conjectured in \cite{D} that the J-equation is solvable under the assumption that
\begin{equation}\label{e sub}
\frac{n [\chi]\cdot[\omega_X]^{n-1}}{[\omega_X]^n} [\omega_X] -[\chi] >0.
\end{equation}

Chen \cite{C2} confirmed Donaldson's conjecture when $n=2$. Later on, Weinkove-Song \cite{SW} proved that the existence of solution to J-equation is equivalent to a subsolution condition: there exists $\underline{u}$ and a constant $\delta>0$ such that
\begin{equation}\label{e sub0}
n \omega_{\underline{u}}^{n-1} \ge (1+\delta) C_{[\omega],[\chi]} (n-1)\omega_{\underline{u}}^{n-2} \wedge \chi.
\end{equation} 
This subsolution condition was replaced by a stability condition formulated by Lejmi-Székelyhidi in \cite{LS} as a necessary and sufficient condition.  It is first proved by G. Chen \cite{CG} and later by Song \cite{S1}. Interested reasders may go to \cite{CG} and \cite{S1} for more complete and updated references.

\subsection{Poincar\'e type K\"ahler metric}
In this paper, we study the J-equation for Poincar\'e type K\"ahler metrics and its relation with K-energy. On $\mathbb{C}^n$, we can write down the standard local model for the Poincar\'e type K\"ahler metric: 
\begin{equation}\label{standard cusp}
    \omega_0= \Sigma_{j=1}^k \sqrt{-1}\frac{ 2 dz^j \wedge d\bar z^j}{|z^j|^2 log^2(|z^j|^2)} +\Sigma_{j=k+1}^{n}\sqrt{-1}dz^j \wedge d \bar z^j.
\end{equation}
Write $z^j$ in the form of polar coordinates, we have that:
\begin{equation*}
    z^j =r_j e^{i\theta_j}.
\end{equation*}
Then we can take $t_j = \log (-2 \log r_j)$. Then we can write $\omega_0$ as:
\begin{equation}\label{standard cusp polar}
    \omega_0 = \Sigma_{j=1}^k- 2e^{- t_j} dt \wedge d\theta_j +\Sigma_{j=k+1}^{n}\sqrt{-1}dz^j \wedge d \bar z^j. 
\end{equation}
Let $D=\Sigma_{j=1}^k D_j$ be a divisor on $X$ as a decomposition of irreducible components. Assume that each $D_j$ is smooth except for some simple normal  self intersections, and  $D_i$ only have simple normal crossings with $D_j$ for $i \neq j$. Let $M$ be the complement of $D$. We follow notation in \cite{A}:
\begin{defn}\label{def 1.1}
    We say that $\omega$ is a Poincar\'e type K\"ahler metric, of class $\Omega=[\omega_X]_{dR}$, denoted by $\omega \in \mathcal{PM}_{\Omega}$, if for any point $p\in D$, and any holomorphic coordinate $U$ of $X$ around $p$ such that in the coordinate $D= \cup_{j=1}^k \{z_j=0\}$( We call this kind of coordinate cusp coordinate from now on), $\omega$ satisfies:
    \begin{enumerate}
        \item There exists a constant $C$ such that $\frac{1}{C}\omega_0 \le \omega \le C \omega_0$ holds in $U$.
        \item There exists a function $\varphi$ such that $\omega=\omega_X +dd^c \varphi$. There exists a constant $C(l)$ such that in $U$, $|\nabla^l_{\omega_0}\varphi|_{\omega_0}\le C(l)$ for any $l\ge 1$. Moreover, $\varphi=\Sigma_{j=1}^k O(log(-log|z_j|))$.
        \item $\omega$ is a smooth K\"ahler metric on $M$.
    \end{enumerate}
\end{defn}
\begin{rem}
    In our definition, we assume that $D_j$ could possibly have self-intersections while in Auvray's definition, $D_j$ doesn't have self-intersections. Since the definition of weighted Sobolev and H\"older  norms are purely local which is not affected by the self-intersections, the main difficulty as far as we can see is the construction of a background Poincar\'e type K\"ahler metric.  For example, assume that $S$ is a defining section of $D_1$ and near a self-intersection point of $D_1$, $S=u z_1 z_2$ for some nonzero function $u$. we can see that $ -\log (\lambda -\log (|S|^2))$ can't be used to define a Poincar\'e type metric near $\{z_1=z_2=0\}$. Our idea is that we can glue $ -\log (\lambda -\log (|S|^2))$ with a model Poincar\'e type metric near the self-intersection point sets and get a background Poincar\'e type metric. The details can be seen in the section 2.1.
\end{rem}

In this paper, $D$ is always a divisor with possibly  simple normal crossing singularities (including self-intersection points) unless otherwise explicitely stated.
\subsection{Main results}

The main theorem that we prove in this paper is:
\begin{thm}\label{main thm}
Let $(X,D)$ be a closed K\"ahler manifold $X$ with a divisor $D\;$ which may have simple normal crossings and self-intersections. Let $\omega$ and $\chi$ be two Poincar\'e type K\"ahler metrics on $X$ with singularity at $D$.  Then (\ref{e j}) admits a unique solution $\omega_{\varphi}$ in $\mathcal{PM}_{[\omega]_{dR}}$ 
if and only if there exists a subsolution $\omega_{\underline{u}} \in \mathcal{PM}_{[\omega]_{dR}}$ satisfying (\ref{e sub0}) and an asymptotic growth condition  

\begin{equation}\label{sub asym}
\frac{\omega_{\underline{u}}^{n}}{\omega_{\underline{u}}^{n-1}\wedge \chi}= C_{[\omega],[\chi]} +O(\rho^{-\eta}).
\end{equation}

\end{thm}

Clearly, this can be viewed as generalization of Song-Weinkove theorem \cite{SW} to the cusp settings. This addtional constraints (\ref{sub asym}) is needed
for complete manifolds. This goes back to the study of complete Tian–Yau metrics \cite{TY} \cite{CL}, in which a background metric is assumed to be asymptotically Ricci-flat. Similar asymtoptic condition is also used by Auvray \cite{A} in solving Poincar\'e type complex Monge-Amp\'ere equations.
Roughly speaking, in a cusp coordinate around $D$ in which $D=\cup_{i=1}^k \{z_i=0\}$, we have that $\rho \approx \Pi_{i=1}^k \frac{1}{-\log |z_i|^2}$. In the case that $k=1$, $-\log \rho$ is comparable to a distance function defined by a background Poincar\'e type metric (c.f. section 2.1 for precise definition). (\ref{sub asym}) means that asymptotically the subsolution solves J-equation. \\

On K\"ahler surfaces, any smooth divisor is precisely a compact Riemann surface, so we can solve the J-equation on it for free. 
Then we can use the subsolution condition in the smooth setting to construct Poincar\'e type metric satisfying (\ref{e sub0}) and (\ref{sub asym}), which implies the solvability of Poincar\'e type J-equation. Consequently, we prove generalized Chen's theorem \cite{C2}  on Donaldson's conjecture in $(X,D)$ at dimension 2.
\begin{thm}\label{prop 2d}
    Let $X$ be a closed K\"ahler surface. Let $D$ be a smooth divisor on $X$. Let $[\omega_0]$ and $[\chi_0]$ be two K\"ahler classes on $X$ satifying  Donaldson's condition (\ref{e sub}), i.e, $2\omega_0 > C_{[\omega_0],[\chi_0]} \chi_0,$ then there exist  Poincar\'e type metrics $\omega_1 \in \mathcal{PM}_{[\omega_0]}$ and $\chi_1\in \mathcal{PM}_{[\chi_0]}$ such that
\begin{equation*}
\omega_1^2 =C_{[\omega_0],[\chi_0]} \omega_1\wedge \chi_1.
\end{equation*}
\end{thm}

\begin{cor}\label{negative self}
Let $X$ be a K\"ahler surface with no curves of negative self-intersections. Let $D$ be a smooth divisor on $X$. Suppose that $K_X[D]$ is ample. Then for any K\"ahler class $[\omega]$, there exist Poincar\'e type K\"ahler metrics $\omega_1\in \mathcal{PM}_{[\omega]}$ and $\chi \in \mathcal{PM}_{K_X[D]}$ such that 
    \begin{equation*}
        \omega_1^2 = C_{[\omega_0],K_X[D]} \omega_1 \wedge \chi.
    \end{equation*}    
\end{cor}
The assumptions of the Corollary \ref{negative self} implies that $2\omega_0 > C_{[\omega_0],[K_X[D]]} K_X[D]$. Then the Corollary \ref{negative self} follows from the Theorem \ref{prop 2d}.

We prove that there exists $\chi \in \mathcal{PM}_{K_X[D]}$ which can be written as $\chi=Ric(\omega_1)$ for some $\omega_1 \in \mathcal{PM}_{[\omega]}$ and at the same time can be used as in Corollary \ref{negative self}. Then the following Corollary follows from the Corollary \ref{negative self}:
\begin{cor}\label{thm lower k}
   Let $X$ be a K\"ahler surface with  no curves of negative self-intersections. Let $D$ be a smooth divisor on $X$. Suppose that $K_X[D]$ is ample. Then for any K\"ahler class $[\omega]$, K-energy is bounded from below on $\mathcal{PM}_{[\omega]}$.
\end{cor}

\begin{defn}
   Suppose that $D=\Sigma_i D_i$, where $D_i$ are smooth complex codimension one manifolds and $D_i$ doesn't intersect with $D_j$ for any $i\neq j$. We say that a Poincar\'e type K\"ahler metric $\omega$ is locally asymptotic to a product metric near $D$, if for each $i$, there exists positive constants $a_i$ and $\eta$ and a K\"ahler metric $\omega_{D_i}$ on $D_i$ such that 
   \begin{equation*}
       \omega= p^* \omega_{D_i} +2a_i e^{-t} dt \wedge \widetilde{\eta} +O(e^{-\eta t}).
   \end{equation*}
   The definition of $p$, $t$ and $\widetilde{\eta}$ are given in the section 2.5.
\end{defn}

Then, we prove the following asymptotic behaviour of solution to Poincar\'e type J-equation which is of its own interest.
\begin{thm}\label{thm asy}
Let $D=\Sigma_i D_i$ be a disjoint union of smooth divisors. Suppose that $\chi$ and $\omega$ are both Poincar\'e type K\"ahler metrics. Suppose that $\chi$ is locally asymptotic to a product metric near each $D_i$:
\begin{equation*}
    \chi= p^* \chi_{D_i} + 2b_i e^{-t} dt \wedge \widetilde{\eta} +O(e^{-\eta t}).
\end{equation*}
Suppose that $\omega_{\varphi}$ satisfies (\ref{e j}).
Then for each $i$, $\omega_{\varphi}$ is locally asymptotic to a product metric near $D_i$:
\begin{equation*}
\omega_{\varphi}= p^* \omega_{D_i} +2a_i e^{-t}dt \wedge \widetilde{\eta} + O(e^{-\eta t}).
\end{equation*}
Moreover, $\omega_{D_i}$ solves a J-equation on $D_i$:
\begin{equation*}
    tr_{\omega_{D_i}}\chi_{D_i} = n-\frac{b_i}{a_i}.
\end{equation*}
\end{thm}
\begin{rem}
 We expect that the above theorem also holds for $D$ in the Theorem \ref{main thm}.
\end{rem}
The asymptotic behaviour of Poincar\'e type extremal K\"ahler metrics  was proved by Auvray in \cite{A2}.

\subsection{Key ingredients}
The traditional continuity path for J-equation is:
\begin{equation}\label{e tra}
    \omega_{\varphi_t}^n = e^{C_t + tf}\omega_{\varphi_t}^{n-1}\wedge \chi,
\end{equation}
where $e^f= \frac{\omega^n}{\omega^{n-1}\wedge \chi}$ and $C_t$ is some constant depending on $t$. 
There are two difficulties in using this continuity path in the Poincar\'e type case:

\begin{enumerate}
\item Although one may assume conditions (\ref{e sub0}) and (\ref{sub asym}) at $t=0$, it is generally difficult to verify that these conditions persist for varying $t$ without an explicit understanding of the constant $C_t$. A characterization of $C_t$ was given by Guo--Song \cite{GS}; however, in general, no explicit formula for $C_t$ is available. This issue is particularly important in the study of uniqueness of viscosity solutions for general complex Hessian equations on Hermitian manifolds (cf. \cite{CX}). It also plays a crucial role in the Poincaré-type setting when one seeks to show that (\ref{e sub0}) and (\ref{sub asym}) are preserved along the continuity path.

\item It is difficult to establish openness for complete cusp metrics when $D$ has simple normal crossings, since the linearized operator may fail to have a bounded inverse. Consequently, one needs to understand the asymptotic behavior of solutions to the $J$-equation near $D$. Such behavior can be derived when $D$ is smooth without singularity (see Theorem \ref{thm asy}). However, in the presence of simple normal crossings or simple normal self-intersections, the asymptotic behavior remains unknown.
\end{enumerate}

To overcome these difficulties, we propose a new two-parameters continuity path:

\begin{equation}\label{new path}
\omega_{\varphi_{\epsilon,t}}^n = C_{[\omega],[\chi]} e^{\epsilon \varphi_{\epsilon,t}}  \omega_{\varphi_{\epsilon,t}}^{n-1} \wedge (t\chi+(1-t)\omega), \,\,\,\,\,\,\,\,\, (*_{\epsilon,t})
\end{equation}
with $\epsilon>0$. Here $\omega$ is the subsolution satisfying (\ref{e sub0}) and (\ref{sub asym}). For $t=0$, (\ref{new path}) always has a solution $\varphi_{\epsilon,0}=0$. First, we prove that if $\epsilon \varphi_{\epsilon,t}\ge -\delta_0$ for some uniformly small constant $\delta_0$ on $M\setminus D$, then we have that $||\varphi_{\epsilon,t}||_{L^{\infty}}\le C$ for some uniform constant $C$, see the Proposition \ref{prop improve linfty}. When $\epsilon$ is small we have that $\frac{\delta_0}{\epsilon} >> C$. So this can be seen as an improvement of $L^{\infty}$ estimate of $\varphi_{\epsilon}$. Then we can use this result to prove uniform $L^{\infty}$ estimate for $\varphi_{\epsilon,t}$, independent of $\epsilon.$ This is one of the key ideas of this paper. The uniform $L^{\infty}$ estimate for $\varphi_{\epsilon,t}$ implies that $|\epsilon \varphi_{\epsilon,t}|$ is small when $\epsilon$ is small. This implies higher order a priori estimate for $\varphi_{\epsilon,t}$. Then we can let $\epsilon$ go to zero and the limit of $\varphi_{\epsilon,t}$ solves Poincar\'e type J-equation.

There are two motivations of this new continuity path which correspond to the two difficulties in using traditional continuity path as we explained above:\\
\begin{enumerate}

    \item We use $(t\chi+(1-t)\omega)$ in (\ref{new path}) which is inspired by Chen's continuity path to solve a cscK metric (c.f. \cite{C1}).  In order to explain our idea better, we ignore the $e^{\epsilon \varphi_{\epsilon,t}}$ in (\ref{new path}) for now and look at the following continuity path:
\begin{equation}\label{new path 2}
    \omega_{\varphi_{t}}^n = C_{[\omega],[\chi]} e^{C_t}\omega_{\varphi_t}^{n-1}\wedge (t\chi+ (1-t)\omega).
\end{equation}
One advantage of using $(t\chi+(1-t)\omega)$ in the continuity path is that we have a very nice characterization of $C_t$ and this is a key point when we want to have a subsolution satisfying an asymptotic behaviour near $D$ along the continuity path. Comparing with the argument by Guo-Song \cite{GS}, it is also easier to show that we always have a subsolution along continuity path.

\item  The part $e^{\epsilon \varphi_{\epsilon,t}}$ in (\ref{new path}) is motivated by Tian-Yau's work on complete Ricci flat metric \cite{TY1} and Auvray's work on solving Poincar\'e type Monge-Amp\'ere equation \cite{A}. The advantage of this method is that it is easier to study the linearized operators to get openness of continuity path, since the extra term $e^{\epsilon \varphi_{\epsilon,t}}$ helps with $L^{\infty}$ estimate for linearized operator.  
\end{enumerate}

\subsection{Constants normalization}In the rest of the paper, up to multiplying $\chi$ and $\omega$ by an appropriate constant we can assume that
\begin{equation}\label{normalize}
\int_X \chi \wedge \omega^{n-1}= \int_X \omega^n=1
\end{equation}
which implies that
\begin{equation}\label{norm c}
    C_{[\omega],[\chi]}=1.
\end{equation}
Then the constant $C_t$ in (\ref{new path 2}) can be expressed as 
\begin{equation*}
e^{C_t} = \frac{\int_X \omega_{\varphi_t}^n}{ \int_X (t\chi + (1-t)\omega) \wedge \omega^{n-1}} = \frac{\int_X \omega^n}{ \int_X t \chi \wedge \omega^{n-1} + (1-t) \int_X \omega^n}=  1
\end{equation*}
So (\ref{new path 2}) becomes
\begin{equation}\label{new path1}
\omega_{\varphi_t}^n =  \omega_{\varphi_t}^{n-1} \wedge (t\chi + (1-t)\omega).
\end{equation}
Note $\omega$ is a subsolution satisfying the subsolution condition (\ref{e sub0}) and asymptotic behavior (\ref{sub asym}). We will show that $\omega$ also satisfy the subsolution condition and asymptotic behavior with respect to (\ref{new path1}). Indeed, since $\omega_{\underline{u}}$ satisfies
\begin{equation*}
n \omega^{n-1}> (1+\delta)(n-1) \omega^{n-1},
\end{equation*}
 we have that
\begin{equation*}
n \omega^{n-1} \ge (1+\delta) (n-1) \omega^{n-2} \wedge (t \chi + (1-t) \omega).
\end{equation*}
We also have that $\underline{u}$ satisfies the asymptotic behavior:
\begin{equation*}
\frac{\omega^n}{\omega^{n-1}\wedge (t \chi + (1-t)\omega)} = 1+ O(\rho^{-\eta }).
\end{equation*}

\subsection{Organzation}In section 2, we introduce some background knowledge about Poincar\'e type K\"ahler metrics. In section 2, we prove the openness of continuity path (\ref{new path}). In section 4, we prove the Proposition \ref{c11} which is an improvement of $L^{\infty}$ estimate. Then, we prove that uniform $L^{\infty}$ estimate implies higher order estimates. In section 5, we prove the Theorem \ref{main thm}. In section 6, we prove the Theorem \ref{prop 2d}, Corollary \ref{negative self} and the Corollary \ref{thm lower k} for K\"ahler surfaces. In section 7, we prove the Theorem \ref{thm asy}.\\


{\bf Acknowledgement}
The second author thanks Junbang Liu for discussion about Poincar\'e type K\"ahler metrics in the case that $D$ has simple normal crossings.
\section{Background knowledge for Poincar\'e type K\"ahler metrics}
\subsection{Background metric of Poincar\'e type}
Denote $D=\Sigma_{j=1}^k D_j$.  First, we briefly describe Auvray's construction \cite{A} of a background Poincar\'e type K\"ahler metric in the case that $D_j$ doesn't have self-intersections for any $j$.  Take a holomorphic defining section $S_j \in (\mathcal{O}([D_j]),|\cdot|)$ for $D_j$. Define 
$$\rho_j \triangleq -\log(|S_j|^2)\ge 1,\,\,\, \rho \triangleq \Pi_{j=1}^k \rho_j$$ 
out of $D_j$, equivalently, $|S_j|^2\leq e^{-1}$. 
Let $\lambda$ be a nonnegative real constant to be determined. Set 
$$\mathbf{u}\triangleq \Sigma_{i=1}^k \log(\lambda+\rho_j).$$ 
Auvray shows in \cite[Lemma 1.1]{A} that for any $A>0$ and for sufficiently large $\lambda$ depending on $A$ and $\omega_X$, the $(1,1)$-form $\omega_X-Ai \partial \bar \partial \mathbf{u}$ is a Poincar\'e type K\"ahler metric. 

Now we want to generalize the above construction to the case where $D_j$ may have self intersections. Compared with Auvray's construction, we have to let the constant $A$ be small enough.  We consider the following easiest case first: Assume that the dimension of $M$ is two and $D$ is an irreducible divisor on $M$ with one simple normal self intersection point $p$. Take a local coordinate $(z_1,z_2)$ around $p$ such that $D=\{(z_1,z_2): z_1=0 \text{ or } z_2=0\}$. Take a cut-off function $\chi$ such that $\chi=1$ on $D_{\epsilon} \triangleq \{(z_1,z_2): |z_1|^2 \le \epsilon, |z_2|^2 \le \epsilon\}$ and $\chi$ is supported in $D_{2\epsilon}$ which is contained in the above coordinate $(z_1,z_2)$.  Define function $\varphi_1 \triangleq -\log (-\log |z_1|^2)- \log (-\log |z_2|^2)$ in this coordinate chart. Define $\varphi_2 \triangleq -\log (\lambda- \log |S|^2)$ where $S \in (\mathcal{O}([D]),|\cdot|)$ is a defining section for $D$. 

Claim: $\omega_X + A dd^c (\chi \varphi_1 + (1-\chi)\varphi_2)$ is a Poincar\'e type metric for positive $A$ that is small enough. 

Proof of the Claim: We compute that
\begin{equation}\label{back 2}
\begin{split}
    &\omega_X + A dd^c (\chi \varphi_1 + (1-\chi)\varphi_2) =\omega_X + A   dd^c (\varphi_2 + \chi (\varphi_1 - \varphi_2)) \\
    & = \omega_X + A dd^c \varphi_2 + A dd^c \chi  +A \chi dd^c (\varphi_1 -\varphi_2)+ A d\chi \wedge d^c (\varphi_1 -\varphi_2)+ A d(\varphi_1- \varphi_2)\wedge d^c \chi \\
    & =\omega_X + A (1-\chi)dd^c \varphi_2 + A \chi dd^c \varphi_1 + A (\varphi_1-\varphi_2) dd^c \chi  + A d\chi \wedge d^c (\varphi_1 -\varphi_2) + A d(\varphi_1 -\varphi_2) \wedge d^c \chi
\end{split}
\end{equation}
Note that in the above coordinate around $p$, we can write $|S|^2=|z_1|^2 |z_2|^2 e^F$. Then we have that 
\begin{equation*}
    \varphi_2 = -\log (\lambda -\log |z_1|^2 -\log |z_2|^2 -F).
\end{equation*}
Note that $d\chi$, $d^c \chi$ and $dd^c \chi$ are supported in 
\begin{equation*}
    D_{2\epsilon}\setminus D_{\epsilon} \subset \{(z_1,z_2): \epsilon \le |z_1| \le 2\epsilon, |z_2|\le 2\epsilon\} \cup \{(z_1,z_2): \epsilon \le |z_2| \le 2\epsilon, |z_1|\le 2\epsilon\}
\end{equation*}
In $\{(z_1,z_2): \epsilon \le |z_2| \le 2\epsilon, |z_1|\le 2\epsilon\}$,
we can compute that 
\begin{equation*}
\begin{split}
& d(\varphi_2 - \varphi_1)= -d \log (1+\frac{-\lambda + \log |z_2|^2 + F}{\log |z_1|^2}) + d \log (-\log |z_2|^2) \\
&=-\frac{1}{1+\frac{-\lambda +\log |z_2|^2 +F}{\log |z_1|^2}} (\frac{d \log |z_2|^2 +dF}{\log |z_1|^2} -\frac{(-\lambda +\log|z_2|^2 +F)d \log |z_1|^2}{(\log |z_1|^2)^2}) + d \log (-\log |z_2|^2)\\
&= O(\frac{1}{|\log |z_1||}) + d \log (-\log |z_2|^2).
\end{split}
\end{equation*}
Here $d \log (-\log |z_2|^2)$ is a smooth metric in $\{(z_1,z_2): \epsilon \le |z_2| \le 2\epsilon, |z_1|\le 2\epsilon\}$. 
Similarly, we can get that 
\begin{equation*}
    d(\varphi_2 - \varphi_1)= O(\frac{1}{|\log |z_2||}) + d \log (-\log |z_1|^2)
\end{equation*}
in $\{(z_1,z_2): \epsilon \le |z_1| \le 2\epsilon, |z_2|\le 2\epsilon\}$ and $d \log (-\log |z_1|^2)$ is a smooth metric in this region.

This implies that 
\begin{equation}\label{dchi}
    d \chi \wedge d^c (\varphi_1- \varphi_2) +d (\varphi_1- \varphi_2) \wedge d^c \chi = \omega_1+ O(\frac{1}{\log |z_1|}+ \frac{1}{\log |z_2|})
\end{equation}
for some smooth form $\omega_1$.
We can also compute similarly that
\begin{equation}\label{ddcchi}
    (\varphi_1- \varphi_2) dd^c \chi= \omega_2 + O(\frac{1}{\log |z_1|}+ \frac{1}{\log |z_2|})
\end{equation}
for some smooth form $\omega_2$. Note that $\omega_X+ A dd^c \varphi_2$ is a Poincar\'e type K\"ahler metric on $X\setminus D_{\epsilon}$ and $\omega_X+ A dd^c \varphi_1$ is a Poincar\'e type K\"ahler metric on $D_{2\epsilon}$, if we let $A$ be small enough. Then we can use (\ref{back 2}), (\ref{dchi}) and (\ref{ddcchi}) to get that $\omega_X + A dd^c (\chi \varphi_1 + (1-\chi)\varphi_2)$ is a Poincar\'e type metric on $X$ if we let $A$ be small enough. Next, we show how to generalize the above construction to general case. Consider a finite covering $\{U_i\}_{i=1}^N$ of $D$ satisfying that for each $i$, $U_i \cap D= \cup_{l=1}^{k_i}\{z_l=0\}.$ Take an open set $U_{N+1}\subset X\setminus D$ such that $X\subset \cup_{i=1}^{N+1} U_i$. Then we take a partition of unity $\chi_i$ with respect to this covering. For each $i \le N$, we define $\varphi_i = \Sigma_{l=1}^{k_i} -\log (-\log |z_l|^2)$. Then, we can define $\varphi = \Sigma_{l=1}^{N} \chi_l \varphi_l$. Then  we can get that $\omega_X + A dd^c \varphi$ is a Poincar\'e type metric on $X$ if we let $A$ be small enough using similar calculation as in the above example.

We define $\rho \triangleq \Sigma_i \chi_i \rho_i$, where $\rho_i=\Pi_{l=1}^{k_i}(-\log |z_l|^2)$ is defined in $U_i$.

We define the space of Poincar\'e type K\"ahler metrics of class $\Omega$ as $ \mathcal{PM}_{\Omega}$ using the Definition \ref{def 1.1}. Denote the space of potentials (with respect to the background metric $\omega$) as
$\widetilde{\mathcal{PM}}_{\Omega}$.

\subsection{Quasi coordinates}
Next, the quasi coordinates, see \cite{TY},  is used to define function spaces using Poincar\'e type K\"ahler metrics.
Let $\Delta$ be a unit disc and let $\Delta^*$ be a punctured unit disc. For any $\delta\in (0,1)$, we can set $$\varphi_{\delta}: \frac{3}{4}\Delta \rightarrow \Delta^*,\quad\xi \mapsto exp(-\frac{1+\delta}{1-\delta}\frac{1+\xi}{1-\xi}).$$ For any $\delta =(\delta_1,...,\delta_k) \in (0,1)^k$, we can define \begin{align*}
\Phi_{\delta}: \mathcal{P} \triangleq (\frac{3}{4}\Delta)^k \times \Delta^{n-k}  &\rightarrow  (\Delta^*)^k \times \Delta^{n-k} ,\quad \delta \in (0,1),\\
(z_{1},...,z_{k}, z_{k+1},...,z_n)&\mapsto (\phi_{\delta_1}(z_1),...,\phi_{\delta_k}(z_{k}),z_{k+1},...,z_n).
 \end{align*} 
 We say that a holomorphic coordinate of $X$ is a cusp coordinate if in this coordinate we have $D=\{z_n=0\}$. Let $\omega_0$ be the standard Poincar\'e type metric given by (\ref{standard cusp}). We can directly compute that
 \begin{equation*}
     \Phi_{\delta}^* \omega_0= \Sigma_{i=1}^k \frac{\sqrt{-1}dz^i \wedge d \bar z^i }{(1-|z_i|^2)^2} + \Sigma_{i=k+1}^n \sqrt{-1}dz^i \wedge d\bar z^i
 \end{equation*}
 which is uniformly quasi-isometric to the Euclidean metric independent of $\delta$. Since any Poincar\'e type metric $\omega$ is quasi-isometric to $\omega_0$ in each cusp coordinate, we have that $\Phi^*_{\delta}\omega$ is quasi-isometric to Euclidean metric in each cusp coordinate.

\subsection{Function spaces}

Auvray \cite{A} used the idea of Tian-Yau \cite{TY} to define the following function spaces:
\begin{defn}
If $U$ is a polydisc neighborhood of $D$ with $U \cap D$ given by $\{z_1=z_2=...=z_k=0\}$, we define for $f\in C^{p,\alpha}_{loc}(U \setminus D), (p,\alpha)\in \mathbb{N}\times [0,1),$
\begin{equation*}
||f||_{C^{p,\alpha}(U \setminus D)} \triangleq \sup_{\delta \in (0,1)^k}||\Phi_{\delta}^* f||_{C^{p,\alpha}(\mathcal{P})},
\end{equation*}
assuming that $U \subset   (c\Delta)^k \times \Delta^{n-k}.$ 

Then given a finite number of such open sets $U\in \mathcal{U}$, covering $D$ and an open set $V \subset \subset X \setminus D$ such that $X=V \cup \bigcup_{U \in \mathcal{U}} U$ and a partition of unity $\{\chi_V\} \cup \{\chi_U: U \in \mathcal{U}\},$ we can define the H\"older space 
\begin{equation*}
C^{p,\alpha}(M) \triangleq \{f\in C_{loc}^{p,\alpha}(M) : ||\chi_V f||_{C^{p,\alpha}(V)}+\max_{U\in 
 \mathcal{U}} ||\chi_U f||_{C^{p,\alpha}(U \setminus D)} < \infty \}.
\end{equation*}
\end{defn}
\begin{defn}
We can define the weighted H\"older norm:
\begin{equation*}
C_{\eta}^{l,\alpha} \triangleq \{f\in C^{l,\alpha}_{loc}(M) : ||\chi_V f||_{C^{l,\alpha}(V)}+\sup_{U\in \mathcal{U}} \sup_{\delta \in (0,1)}||\Pi_{i=1}^k(1-\delta_i)^{\eta}\Phi_{\delta}^* (\chi_U f)||_{C^{l,\alpha}(\mathcal{P})} < \infty\}.
\end{equation*}
Since  $\Pi_{i=1}^k\frac{1}{C(1-\delta_i)} \le \Phi_{\delta}^* \rho \le \Pi_{i=1}^k\frac{C}{1-\delta_i}$ for some constant $C$, $||\Pi_{i=1}^k(1-\delta_i)^{\eta}\Phi_{\delta}^* (\chi_U f)||_{C^{l,\alpha}(\mathcal{P})} $ is equivalent to $||\Phi_{\delta}^* (\rho^{-\eta} \chi_U f)||_{C^{k,\alpha}(\mathcal{P})} $.
Heuristically, $f\in C_{\eta}^{l,\alpha}$  implies that $f =O(\rho^{\eta})$. We can also define:
\begin{equation*}
    C_{\eta}^{\infty}= \cap_{l=0}^{\infty} C_{\eta}^{l,\alpha}.
\end{equation*}
\end{defn}
Then we can define 
\begin{equation*}
    C^{\infty}(M)= \cap_{k=1}^{\infty}C^{k,\alpha}(M).
\end{equation*}

\begin{defn}
 We can also define the weighted Sobolev space:
\begin{equation*}
W_{\eta}^{k,p} \triangleq \{v \in W^{k,p}_{loc}(M): \int_{M} \Sigma_{i=0}^k |\nabla_i v|^p \rho^{-\eta p}\omega^n < \infty \}.
\end{equation*}
\end{defn}
Clearly, $W_{\eta}^{k,p} \subset W_{\eta'}^{k,p} $, when $\eta\leq \eta'$.

\subsection{Energy functionals}
Next we define several functionals defined on $\widetilde{\mathcal{PM}}_{\Omega}$:
\begin{equation}\label{mathcal E}
    \mathcal{E}(\varphi)\triangleq \int_X \varphi  \Sigma_{j=0}^n \omega_\varphi^{n-j}\wedge \omega^j.
\end{equation}
Given a closed $(1,1)$-form (or current) $T$ which satisfies that $T\le C \omega$ for some constant $C$ and all the covariant derivatives of $T$ defined by $\omega$ are bounded with respect to $\omega$.

\begin{equation*}
    \mathcal{E}^T (\varphi )\triangleq \int_X \varphi  \Sigma_{j=0}^{n-1}\omega_\varphi ^{n-j-1}\wedge \omega^j \wedge T.
\end{equation*}
 Denote $\mu_0=\omega^n$. For any measure $\mu$ which is absolutely continuous with respect to $\mu_0$, we can also define the entropy term:
\begin{equation*}
    H_{\mu_0}(\mu)\triangleq \int_X log(\frac{d\mu}{d\mu_0})d\mu
\end{equation*}
The $K$-energy can be expressed as
\begin{equation}\label{decom kenergy}
    \mathcal{M}(\varphi )\triangleq \frac{\bar R}{n+1}\mathcal{E}(\varphi )-\mathcal{E}^{Ric_{\omega}}(\varphi ) +H_{\mu_0}(\omega_\varphi ^n).
\end{equation}
These functionals are well defined, because of the definition of $\widetilde{\mathcal{PM}}_{\Omega}$.

\subsection{Fiber bundle structure of a neighbourhood of D}
Suppose that $D$ is smooth. Let $S$ be a defining section of $D$ in $\mathcal{O}([D])$. According to \cite[Section 3]{A2}, a neighbourhood of $D$, denoted as $\mathcal{N}_A$, can be seen as a $S^1$ bundle over $[A,\infty) \times D$. This fiber bundle can be written as $$q: \mathcal{N}_A \setminus D \xrightarrow{q=(t,p)} [A,\infty) \times D.$$   The function $t$ is defined in \cite{A2}. We have that $t= \log (-\log |S|^2)$ up to a perturbation which is a $O(e^{-t})$, that is, a $O(\frac{1}{|log |S||})$, as well as its derivatives of any order with respect to Poincar\'e metrics. Denote $p$ as the projection from $\mathcal{N}_A \setminus D$ to $D$. We can also define a connection $\widetilde{\eta}$ in $\mathcal{N}_A \setminus D$ which can be seen as a  a 1-form on each $S^1$ fibre such that  $J$, as complex structure on $X$, satisfying

$$J dt =2e^{-t} \widetilde{\eta} +O(e^{-t}).$$
In a cusp coordinate $(z_1,...,z_n=r e^{i\theta})$, one has 
\begin{equation}\label{eta dtheta}
\widetilde{\eta} =d \theta +O(1)
\end{equation}
in the sense that $\widetilde{\eta}- d\theta$ and all the derivatives of it of any order with respect to $\omega$ is bounded. 
Given an arbitrary function $f$ supported in a neighbourhood $\mathcal{N}_A$ of $D$, we can decompose $f$ as:
\begin{equation}\label{f decom}
f= f_0(t,p) + f^{\bot},
\end{equation}  
where $$f_0(t,p)=\frac{1}{2\pi}\int_{q^{-1}(t,p)} f \widetilde{\eta}$$ is the $S^1$ invariant part and $f^{\bot}$ is the part that is perpendicular to $S^1$ invariant functions using the integral on each $S^1$ fiber with volume form $\widetilde{\eta}$.

\subsection{Basic Inequality}

Auvray proved the following Poincar\'e Inequality for Poincar\'e type K\"ahler metrics (c.f.  \cite[Lemma 1.11]{A}):
\begin{lem}\label{poincare ine}
Assume $X$ is equipped with a Poincar\'e type K\"ahler metric $\omega$. Then there exists a constant $C_P>0$ such that for all $v\in W^{1,2}_0 (X,\omega)$, we have 
\begin{equation*}
\int_{X}|v-a|^2  \omega^n \le C_P \int_{X}|dv|_{\omega}^2 \omega^n,
\end{equation*}
where $a=\int_{X}v \omega^n$.
\end{lem}

Although the standard Sobolev Inequality for $\omega$ is not yet, Auvray proved the following weighted Sobolev Inequality (See the Lemma 4.4 in \cite{A});
\begin{lem}\label{weighted sobolev}
Let $\omega$ be a Poincar\'e type K\"ahler metric. Then for any $p>0$ and $q>0$ with $\frac{1}{q}< \frac{1}{p}+\frac{1}{2n}$, there exists a constant $C$ such that 
\begin{equation*}
||u||_{W^{k,p}_{-\frac{1}{p}}} \le C ||u||_{W^{k+1,q}_{-\frac{1}{q}}}.
\end{equation*}
\end{lem}

\subsection{Stokes Theorem}
Throughout this paper, we need Gaffney's Stokes theorem \cite{MR0062490} to do integration by parts using Poincar\'e type K\"ahler metrics:
\begin{lem}\label{G-S}
Let $(X,g)$ be a complete n-dimensional Riemannian manifold where $g$ is a $C^2$ metric tensor. Let $\Theta$ be a $C^1$ $(n-1)$ form on $M$  such that both $|\Theta|_g$ and $|d\Theta|_g$ are in $L^1$. Then we have that $\int_X d\Theta = 0$.
\end{lem}

\section{Openness of continuity path}

In this section, we will fix $\epsilon>0$ and prove the openness of the continuity path $(\ref{new path})$ in $t$ direction. Denote $\chi_t = t\chi+(1-t)\omega$ Then the linearized operator is
\begin{equation*}
    Lu= g_{\varphi_{\epsilon,t}}^{i \bar j}(\chi_{t})_{\bar j l}g_{\varphi_{\epsilon,t}}^{l\bar k} u_{\bar k i}-n \epsilon e^{-\epsilon \varphi_{\epsilon,t}}u.
\end{equation*}
The openness of continuity path follows from the implicit function theory and the following Proposition:
\begin{prop}\label{prop open}
    Suppose that $\omega_{\varphi_{\epsilon,t}}$ is a Poincar\'e type metric. Then, for any $k$, $L$ is an isometry from $C^{k+2,\alpha}(M)$ to $C^{k,\alpha}(M)$.
\end{prop}
\begin{proof}
   The proof of this Proposition is similar to the Lemma 1.3 of \cite{TY}. For completeness we include the proof here. The injective of $L$ follows immediately from the generalized maximum principle of Yau on complete manifolds. For surjective, given a function $f\in C^{k,\alpha}(M)$, we want to find $u\in C^{k+2,\alpha}(M)$ such that $Lu =f.$ Fix a point $p\in M$, we take a sequence of geodesic balls $B_k(p)$ and solve the equation:
   \begin{equation*}
       \begin{split}
           L u_k &= f \text{ on } B_k(p) \\
           u_k & = 0 \text{ on } \partial B_k(p)
       \end{split}
   \end{equation*}
   By maximum principle, we have that 
   \begin{equation*}
       ||u_k||_{L^{\infty}(B_k(p))}\le \frac{||f||_{L^{\infty}(M)}}{C}.
   \end{equation*}
   where $C$ is a positive constant such that $n\epsilon e^{-\epsilon \varphi_{\epsilon,t}} \ge C$. For any quasi coordinate  $(U,\Phi_{\delta})$, we have that $U \subset B_k(p)$ if $k$ is big enough. Then we can apply schauder estimate and bootstrapping argument to $\Phi_{\delta}^* L (\Phi_{\delta}^* u)= \Phi_{\delta}^* f$, using the fact that the coefficients  of $\Phi_{\delta}^* L$ and $\Phi_{\delta}^* f$ are bounded in $C^{k,\alpha}$, we can get that $||u_k||_{C^{k+2,\alpha}(U)}$ is uniformly bounded. Then we can take a diagonal argument to take a subsequence of $u_k$ such that it converge locally uniformly to $u$. Moreover $u\in C^{k+2,\alpha}(M)$ and solves the equation 
   \begin{equation*}
       Lu=f.
   \end{equation*}
\end{proof}

\section{Closedness of continuity path}
Throughout this section, we use the normalization (\ref{normalize}).
\subsection{$L^{\infty}$ estimate}
The main Proposition that we want to prove in this section is as follows. Note that Sun \cite{S} first used pluripotential theory to obtain $L^{\infty}$ estimate for J-equation.
\begin{prop}\label{prop improve linfty}
    Let $\varphi_{\epsilon,t}$ be a solution to (\ref{new path}). Suppose that $\omega$ satisfies the strict subsolution condition:
    \begin{equation*}
n \omega^{n-1} \ge (1+\delta) (n-1) \omega^{n-2} \wedge \chi
\end{equation*}
and
\begin{equation*}
\frac{\omega^{n-1}\wedge \chi}{\omega^{n}}= 1 +O(\rho^{-\eta}).
\end{equation*}
for some $\eta >0$. Then there exist constants $C_0$ and $\delta_0$ depending on $\delta$ such that if $\epsilon \varphi_{\epsilon,t}\ge -\delta_0$, we have that
\begin{equation*}
    ||\varphi_{\epsilon,t}||_{L^{\infty}}\le C_0.
\end{equation*}
\end{prop}

\begin{lem}\label{l2.1}
Assume the same assumption as the Proposition \ref{prop improve linfty}. Then there exists a constants $\delta_0$ and $C$ depending on $\delta$ such that if $\epsilon \varphi_{\epsilon,t}\ge -\delta_0$, we have that for any $p \ge 1$,
\begin{equation}\label{e linfty 4}
\begin{split}
    &\int_{X\setminus D} p |\varphi_{\epsilon,t}|^{p-1} |d \varphi_{\epsilon,t}|^2 \omega^n + \Sigma_{i=1}^{n-2}\int_{X\setminus D} p |\varphi_{\epsilon,t}|^{p-1} \sqrt{-1} \partial \varphi_{\epsilon,t} \wedge \bar \partial \varphi_{\epsilon,t} \wedge \omega_{\varphi_{\epsilon,t}}^i \wedge \omega^{n-2-i} \wedge \chi \\
    &\le C\int_{X\setminus D} \varphi_{\epsilon,t} |\varphi_{\epsilon,t}|^{p-1} (\omega^n - \omega^{n-1}\wedge \chi)
\end{split}
\end{equation}
\end{lem}
\begin{proof}
     Note that each equation $(*_{\epsilon,t})$ in the continuity path (\ref{new path}) has the same form as $(*_{\epsilon,1})$ which only replace $\chi$ by another  K\"ahler form $\chi_t= (1-t)\omega+ t\chi$. Since there is a uniform constant $C$ independent of $t$ such that $1/C \omega \le \chi_t \le C \omega$ and all the derivatives of $\chi$ are bounded with respect to $\omega$, and $\omega$ satisfies the subsolution condition 
\begin{equation}
n \omega^{n-1} \ge (1+\delta) (n-1) \omega^{n-2} \wedge \chi_t
\end{equation}
 according to the discussion in the introduction section, we can just do the estimate for $(*_{\epsilon,1})$ and denote $\varphi_{\epsilon,1}$ as $\varphi_{\epsilon}$ just for simplicity.  The key idea is that in this case $\varphi_{\epsilon}$ is almost a subsolution to $\omega_{\varphi_{\epsilon}}^n = \omega_{\varphi_{\epsilon}}^{n-1}\wedge \chi$. In fact, we can compute that
\begin{equation*}
    \omega_{\varphi_{\epsilon}}^n = e^{\epsilon \varphi_{\epsilon}} \omega_{\varphi_{\epsilon}}^{n-1}\wedge \chi \ge e^{-\delta_0} \omega_{\varphi_{\epsilon}}^{n-1}\wedge \chi.
\end{equation*}
This implies that
\begin{equation}\label{almost sub}
    n \omega_{\varphi_{\epsilon}}^{n-1} \ge e^{-\delta_0} (n-1) \omega_{\varphi_{\epsilon}}^{n-2}\wedge \chi.
\end{equation}

We compute that
\begin{equation}\label{e linfty 1.4}
\begin{split}
&\int_{X\setminus D} \varphi_{\epsilon} |\varphi_{\epsilon}|^{p-1} (-\omega^n + \omega^{n-1}\wedge \chi ) \\
&=\int_{X\setminus D} \varphi_{\epsilon} |\varphi_{\epsilon}|^{p-1} \Big(  (\omega_{\varphi_{\epsilon}}^n -\omega^n )- (e^{\epsilon \varphi_{\epsilon}}\omega_{\varphi_{\epsilon}}^{n-1} \wedge \chi - \omega^{n-1} \wedge \chi) \Big)  \\
& \le \int_{X\setminus D} \varphi_{\epsilon} |\varphi_{\epsilon}|^{p-1} \Big(  (\omega_{\varphi_{\epsilon}}^n -\omega^n )- (\omega_{\varphi_{\epsilon}}^{n-1} \wedge \chi - \omega^{n-1} \wedge \chi) \Big)
\end{split}
\end{equation}
Here in the last line we use that $\varphi_{\epsilon}(e^{\epsilon \varphi_{\epsilon}}-1)\ge 0$.
Note that 
\begin{equation}\label{e lin}
\begin{split}
&(\omega_{\varphi_{\epsilon}}^n -\omega^n )- (\omega_{\varphi_{\epsilon}}^{n-1} \wedge \chi  - \omega^{n-1} \wedge \chi) \\
&=\int_0^1 \sqrt{-1} \partial \bar \partial \varphi_{\epsilon} \wedge \Big(n \omega_{t \varphi_{\epsilon}}^{n-1} - (n-1) \omega_{t\varphi_{\epsilon}}^{n-2} \wedge \chi \Big) dt.
\end{split}
\end{equation}
Let $\sigma_{l}$ be the $l-th$ elementary symmetric polynomial. It is a standard result that $\log(\frac{\sigma_{n-1}(\lambda(A))}{\sigma_{n-2}(\lambda(A))})$ is a concave function over positive symmetric $(n-1) \times (n-1)$ matrices $A$, where $\lambda(A)$ is the vector of the eigenvalues of $A$. Thus we can use (\ref{almost sub}) to compute the pointwise inequality for $1\le i \le n$:
\begin{equation}\label{e sub1}
    \begin{split}
       & \frac{\sqrt{-1} dz^i \wedge d \bar z^i \wedge n \omega_{t \varphi_{\epsilon}}^{n-1}}{\sqrt{-1} dz^i \wedge d \bar z^i \wedge (n-1)\omega_{t \varphi_{\epsilon}}^{n-2}\wedge \chi}  \\
       & \ge e^{t \log (\frac{\sqrt{-1}dz^i \wedge d \bar z^i \wedge n \omega_{\varphi_{\epsilon}}^{n-1}}{\sqrt{-1}dz^i \wedge d \bar z^i \wedge (n-1)\omega_{\varphi_{\epsilon}}^{n-2}\wedge \chi})+ (1-t)\log (\frac{\sqrt{-1} dz^i \wedge d \bar z^i \wedge \omega^{n-1}}{\sqrt{-1} dz^i \wedge d \bar z^i \wedge (n-1)\omega^{n-2}\wedge \chi})} \\
       &  \ge e^{t(-\delta_0) +(1-t)\log (1+\delta)}.
    \end{split}
\end{equation}

Combining (\ref{e linfty 1.4}), (\ref{e lin}) and (\ref{e sub1}), we get
\begin{equation}\label{e lp mos}
\begin{split}
     &-\int_{X\setminus D} \varphi_{\epsilon} |\varphi_{\epsilon}|^{p-1} (-\omega^n + \omega^{n-1}\wedge \chi ) \ge -\int_{X\setminus D} \int_0^1 \varphi_{\epsilon} |\varphi_{\epsilon}|^{p-1} \sqrt{-1} \partial \bar \partial \varphi_{\epsilon} \wedge \Big(n \omega_{t \varphi_{\epsilon}}^{n-1} - (n-1) \omega_{t\varphi_{\epsilon}}^{n-2} \wedge \chi \Big) dt \\
     & =  \int_{X\setminus D} \int_0^1 p |\varphi_{\epsilon}|^{p-1} \sqrt{-1}\partial \varphi_{\epsilon} \wedge \bar \partial \varphi_{\epsilon} \wedge \Big(n \omega_{t \varphi_{\epsilon}}^{n-1} - (n-1) \omega_{t\varphi_{\epsilon}}^{n-2} \wedge \chi \Big) dt \\
     & \ge  \int_{X\setminus D} \int_0^1 p |\varphi_{\epsilon}|^{p-1} \sqrt{-1}\partial \varphi_{\epsilon} \wedge \bar \partial \varphi_{\epsilon} \wedge (e^{-\delta_0 t +(1-t)\log (1+\delta)}-1)(n-1) \omega_{t\varphi_{\epsilon}}^{n-2} \wedge \chi dt \\
     & = \int_{X\setminus D} p |\varphi_{\epsilon}|^{p-1} \sqrt{-1} \partial \varphi_{\epsilon} \wedge \bar \partial \varphi_{\epsilon} \wedge \Sigma_{i=0}^{n-2} C_{\delta_0, i} \omega_{\varphi_{\epsilon}}^i \wedge \omega^{n-2-i}\wedge \chi dt
\end{split}
\end{equation}
In the second line above, we use the Lemma \ref{G-S} to do integration by parts. Note that $C_{0,i}>0$. Thus, there exists $\delta_0$ small such that for any $i$, we have that $C_{\delta_0,i}>0$. 
\end{proof}

We can prove the following estimate:
\begin{lem}\label{l2.2}
Assume the same assumptions as in the Proposition \ref{prop improve linfty}. Define $a_{\epsilon,t}= \int_{X\setminus D}\varphi_{\epsilon,t} \omega^n$ and $b_{\epsilon,t}= \int_{X\setminus D}\varphi_{\epsilon,t} \omega^{n-1}\wedge \chi$. Then there exists a constants $\delta_0$ and $C$ depending on $\delta$ such that if $\epsilon \varphi_{\epsilon,t}\ge -\delta_0$, we have that  $\varphi_{\epsilon,t}$  satisfies that
\begin{equation*}
||\varphi_{\epsilon,t}- a_{\epsilon,t}||_{L^{2}}\le C,\,\,\, ||\varphi_{\epsilon,t}- b_{\epsilon,t}||_{L^{2}}\le C.
\end{equation*}
\end{lem}
\begin{proof}
According to the discussion at the beginning of the proof of the Lemma \ref{l2.1}, we can use $\varphi_{\epsilon}$ instead of $\varphi_{\epsilon,t}.$ Define $a_{\epsilon}= \int_{X\setminus D}\varphi_{\epsilon} \omega^n$ and $b_{\epsilon}= \int_{X\setminus D}\varphi_{\epsilon} \omega^{n-1}\wedge \chi$. Then, we get from Poincar\'e Inequality, i.e. the Lemma \ref{poincare ine}, and the Lemma \ref{l2.1}:
\begin{equation*}
    \int_{X\setminus D}|\varphi_{\epsilon}- a_{\epsilon}|^2 \omega^n \le C \int_{X\setminus D} |d \varphi_{\epsilon}|^2 \omega^n \le C \int_{X\setminus D} \varphi_{\epsilon} (\omega^n - \omega^{n-1}\wedge \chi)= C \int_{X\setminus D} (\varphi_{\epsilon}- a_{\epsilon})(\omega^n- \omega^{n-1}\wedge \chi).
\end{equation*}
Here we use the normalization $\int_{X\setminus D} \omega^n = \int_{X\setminus D}\omega^{n-1}\wedge \chi$.
Thus we can get that
\begin{equation}\label{varphia}
    \int_{X\setminus D}|\varphi_{\epsilon}-a_{\epsilon}|^2\omega^n \le C.
\end{equation} 

Next we estimate $||\varphi_{\epsilon,t}- b_{\epsilon,t}||_{L^{2}}$. First, we use the Theorem 3.2 in \cite{A} and the assumption that $\frac{\omega^n}{\omega^{n-1}\wedge \chi}=1+O(\rho^{-\eta})$ to solve the Poincar\'e type complex Monge-Amp\'ere equation for $\omega'\in \mathcal{PM}_{[\omega]}$:
\begin{equation*}
    \omega'^n = \omega^{n-1} \wedge \chi,
\end{equation*}
using the normalization $\int_{X\setminus D} \omega^n = \int_{X\setminus D}\omega^{n-1}\wedge \chi$. Then we can use Poincar\'e Inequality using metric $\omega'$ and the Lemma \ref{l2.1} to get that:
\begin{equation*}
\begin{split}
       &\int_{X\setminus D}|\varphi_{\epsilon}- b_{\epsilon}|^2 \omega^n \le C \int_{X\setminus D}|\varphi_{\epsilon}- b_{\epsilon}|^2 \omega'^n \le C \int_{X\setminus D} |d \varphi_{\epsilon}|^2 \omega'^n \\
       & \le C \int_{X\setminus D} |d \varphi_{\epsilon}|^2 \omega^n \le C \int_{X\setminus D} \varphi_{\epsilon} (\omega^n - \omega^{n-1}\wedge \chi)= C \int_{X\setminus D} (\varphi_{\epsilon}- b_{\epsilon})(\omega^n- \omega^{n-1}\wedge \chi).
\end{split}
\end{equation*}

Thus, we can get that 
\begin{equation}\label{varphib}
    \int_{X\setminus D}|\varphi_{\epsilon}-b_{\epsilon}|^2\omega^n \le C.
\end{equation} 
\end{proof}

Then we can prove the following $L^2$ estimate:
\begin{lem}\label{l2}
    Assume the same assumptions as in the Proposition \ref{prop improve linfty}. Then there exists a constants $\delta_0$ and $C$ depending on $\delta$ such that if $\epsilon \varphi_{\epsilon,t}\ge -\delta_0$, we have that $\varphi_{\epsilon,t}$ satisfies that
\begin{equation*}
||\varphi_{\epsilon,t}||_{L^{2}}\le C.
\end{equation*}
\end{lem}
\begin{proof}
According to the discussion at the begining of the proof of the Lemma \ref{l2.1}, we can use $\varphi_{\epsilon}$ instead of $\varphi_{\epsilon,t}.$ For simplicity, we can also replace $\omega_{\underline{u}}$ by $\omega$.
Recall that we define $a_{\epsilon}= \int_{X\setminus D}\varphi_{\epsilon} \omega^n$ and $b_{\epsilon}= \int_{X\setminus D}\varphi_{\epsilon} \omega^{n-1}\wedge \chi$.  

We will show that $a_{\epsilon}$ and $b_{\epsilon}$ are uniformly bounded from above and below independent of $\epsilon$. Then the bound on $||\varphi_{\epsilon}||_{L^2(\omega^n)}$ is immediate from the Lemma \ref{l2.2}.

First, we can use Jensen's inequality to get that
\begin{equation*}
    \int_{X\setminus D}\epsilon \varphi_{\epsilon}\omega_{\varphi_{\epsilon}}^{n-1}\wedge \chi \le \log (\int_{X\setminus D} e^{\epsilon \varphi_{\epsilon}}\omega_{\varphi_{\epsilon}}^{n-1}\wedge \chi)=\log (\int_{X\setminus D} \omega_{\varphi_{\epsilon}}^n)=0
\end{equation*}
Here we use the equation for $\varphi_{\epsilon}$ and the normalization that  $\omega^n$ and $\omega^{n-1}\wedge \chi$ are probability measures, so are $\omega_{\varphi_{\epsilon}}^{n-1}\wedge \chi$ and $\omega_{\varphi_{\epsilon}}^n$. Then we have:
\begin{equation}\label{phie u}
\begin{split}
    &0 \ge \int_{X\setminus D}\varphi_{\epsilon}\omega_{\varphi_{\epsilon}}^{n-1}\wedge \chi = \int_{X\setminus D} \varphi_{\epsilon}\omega_{\varphi_{\epsilon}}^{n-2}\wedge \chi \wedge \omega + \int_{X\setminus D}\varphi_{\epsilon}\omega_{\varphi_{\epsilon}}^{n-2}\wedge \chi \wedge \sqrt{-1} \partial  \bar \partial \varphi_{\epsilon}\\
    & = ... = \int_{X\setminus D}\varphi_{\epsilon} \omega^{n-1}\wedge \chi - \Sigma_{i=0}^{n-2}\int_{X\setminus D} \sqrt{-1} \partial \varphi_{\epsilon} \wedge \bar \partial \varphi_{\epsilon} \wedge \omega_{\varphi_{\epsilon}}^{i}\wedge \chi \wedge \omega^{n-i-2}
\end{split}
\end{equation}
Using the Lemma \ref{l2.1} and the Lemma \ref{l2.2}, we can get that:
\begin{equation}\label{gra l2}
    \Sigma_{i=0}^{n-2}\int_{X\setminus D} \sqrt{-1} \partial \varphi_{\epsilon} \wedge \bar \partial \varphi_{\epsilon} \wedge \omega_{\varphi_{\epsilon}}^i \wedge \omega^{n-2-i} \wedge \chi \le C
\end{equation}
Combining (\ref{phie u}) and (\ref{gra l2}), we get that
\begin{equation}\label{bup}
b_{\epsilon}=\int_{X\setminus D}\varphi_{\epsilon}\omega^{n-1}\wedge \chi \le C
\end{equation}
for some uniform constant $C$ independent of $\epsilon.$

Next, we use Jensen's Inequality to get
\begin{equation*}
    1= \int_{X\setminus D} \omega_{\varphi_{\epsilon}}^{n-1}\wedge \chi = \int_{X\setminus D} e^{-\epsilon \varphi_{\epsilon}}\omega_{\varphi_{\epsilon}}^n \ge e^{\int_{X\setminus D}-\epsilon \varphi_{\epsilon}\omega_{\varphi_{\epsilon}}^n}.
\end{equation*}
Thus we have
\begin{equation*}
    \int_{X\setminus D}\varphi_{\epsilon}\omega_{\varphi_{\epsilon}}^n \ge 0.
\end{equation*}
Then we can compute 
\begin{equation}\label{adown}
\begin{split}
       & 0 \le  \int_{X\setminus D}\varphi_{\epsilon}\omega_{\varphi_{\epsilon}}^n = \int_{X\setminus D} \varphi_{\epsilon}\omega_{\varphi_{\epsilon}}^{n-1}\wedge \omega + \int_{X\setminus D}\varphi_{\epsilon} \omega_{\varphi_{\epsilon}}^{n-1}\wedge \sqrt{-1}\partial  \bar \partial \varphi_{\epsilon} \\
       &=\int_{X\setminus D} \varphi_{\epsilon}\omega_{\varphi_{\epsilon}}^{n-1}\wedge \omega- \int_{X\setminus D} \sqrt{-1}\partial \varphi_{\epsilon}\wedge \bar \partial \varphi_{\epsilon} \wedge \omega_{\varphi_{\epsilon}}^{n-1} \le \int_{X\setminus D}\varphi_{\epsilon}\omega_{\varphi_{\epsilon}}^{n-1}\wedge \omega \\
       & \le ... \le \int_{X\setminus D}\varphi_{\epsilon}\omega^{n}=a_{\epsilon}.
\end{split}
\end{equation}

Using the Lemma \ref{l2.2}, we get that
\begin{equation}\label{ab}
    |a_{\epsilon}-b_{\epsilon}|\le ( 2\int_{X\setminus D} |\varphi_{\epsilon}-a_{\epsilon}|^2 \omega^n + 2\int_{X\setminus D}|\varphi_{\epsilon}-b_{\epsilon}|^2 \omega^n)^{\frac{1}{2}} \le C.
\end{equation}
Combining (\ref{bup}), (\ref{adown}) and (\ref{ab}), we get that
\begin{equation*}
    |a_{\epsilon}|\le C
\end{equation*}
for some uniform constant $C$. Combining this with the Lemma \ref{l2.2}, we get that
\begin{equation*}
    \int_{X\setminus D}|\varphi_{\epsilon}|^2\omega^n \le C
\end{equation*}
for some uniform constant $C$.
\end{proof}

\begin{proof}
    (of Proposition \ref{prop improve linfty}.) 
Note that $\frac{|d \rho|_{\omega_0}}{\rho} \le C$ for some uniform constant $C$. Then we can compute
\begin{equation}\label{e weight}
\begin{split}
    & \int_{X\setminus D} |d(\rho^{-\frac{1}{2}}|\varphi_{\epsilon}|^{\frac{p+1}{2}})|^2 \rho \omega^n  +\int_{X\setminus D}|\varphi_{\epsilon}|^{p+1} \omega^n \\
    &\le 2 \int_{X\setminus D} |d|\varphi_{\epsilon}|^{\frac{p+1}{2}}|^2 \omega^n + C \int_{X\setminus D} |\varphi_{\epsilon}|^{p+1}\omega^n \\
    & \le C \int_{X\setminus D} |d|\varphi_{\epsilon}|^{\frac{p+1}{2}}|^2 \omega^n + C (\int_{X\setminus D} |\varphi_{\epsilon}|^{\frac{p+1}{2}}\omega^n)^{2} \\
     & \le \frac{C(p+1)^2}{p} \int_{X\setminus D} \varphi_{\epsilon} |\varphi_{\epsilon}|^{p-1} (\omega^n - \omega^{n-1}\wedge \chi)+ C (\int_{X\setminus D} |\varphi_{\epsilon}|^{\frac{p+1}{2}}\omega^n)^{2}\\
     & \le \frac{C(p+1)^2}{p} \int_{X\setminus D} \varphi_{\epsilon} |\varphi_{\epsilon}|^{p-1} (\omega^n - \omega^{n-1}\wedge \chi)  + C(\int_{X\setminus D}|\varphi_{\epsilon}|^{p+1}\rho^{-\eta} \omega^n)(\int_{X\setminus D} \rho^{\eta} \omega^n) \\
     & \le  \frac{C(p+1)^2}{p} \int_{X\setminus D} \varphi_{\epsilon} |\varphi_{\epsilon}|^{p-1} (\omega^n - \omega^{n-1}\wedge \chi)  + C\int_{X\setminus D}|\varphi_{\epsilon}|^{p+1}\rho^{-\eta} \omega^n
\end{split}
\end{equation}
In the third line above, we use Poincar\'e Inequality for $|\varphi_{\epsilon}|^{\frac{p+1}{2}}$.
Then, we apply the weighted Sobolev Inequality, i.e. the Lemma \ref{weighted sobolev}, to $\rho^{-\frac{1}{2}}|\varphi_{\epsilon}|^{\frac{p+1}{2}}$ with $(p,q)$ in the Lemma \ref{weighted sobolev} to be $(2(1+\eta),2)$. Here we can assume that $\eta$ is small such that $\frac{1}{2}< \frac{1}{p}+\frac{1}{2n}$. Then we get:
\begin{equation*}
\begin{split}
        (\int_{X\setminus D}|\varphi_{\epsilon}|^{(p+1)(1+\eta)} \rho^{-\eta} \omega^n )^{\frac{1}{1+\eta}} &\le C \int_{X\setminus D} |d(\rho^{-\frac{1}{2}}|\varphi_{\epsilon}|^{\frac{p+1}{2}})|^2 \rho \omega^n +C \int_{X\setminus D} |\varphi_{\epsilon}|^{p+1}\omega^n\\
    & \le \frac{C(p+1)^2}{p} \int_{X\setminus D} \varphi_{\epsilon} |\varphi_{\epsilon}|^{p-1} (\omega^n - \omega^{n-1}\wedge \chi) +C\int_{X\setminus D}|\varphi_{\epsilon}|^{p+1}\rho^{-\eta} \omega^n \\
    & \le Cp \int_{X\setminus D} (|\varphi_{\epsilon}|^{p+1}+|\varphi_{\epsilon}|^p) \rho^{-\eta}\omega^n.
\end{split}
\end{equation*}
In the second line of the above formula, we use (\ref{e weight}). In the last line above, we use the assumption that $tr_{\omega}\chi=n+O(\frac{1}{\rho^{\eta}})$. Then we use the standard Moser iteration argument to get that:
\begin{equation*}
||\varphi_{\epsilon}||_{L^{\infty}}\le C (||\varphi_{\epsilon}||_{L^2}+1).
\end{equation*}
Since we have obtained the $L^2$ estimate for $\varphi_{\epsilon}$ in the Lemma \ref{l2}, this concludes the proof of this Proposition. 
\end{proof}

\subsection{High order estimate}

In this section, we want to prove the following Proposition:
\begin{prop}\label{c11}
Suppose there exists a subsolution $\underline{u}$ as in the Theorem \ref{main thm}. Let $C_1$ be a positive constant. Then there exists a constant $\epsilon_0$ depending on $\delta$, $\omega$, $C_1$ and $\chi$. There exist $A$ and $C$ independent of $\epsilon$ such that for any $\epsilon\in (0,\epsilon_0)$ and $t\in [0,1]$, the solution $\varphi_{\epsilon,t}$ to (\ref{new path}), if satisfying $|\varphi_{\epsilon,t}|_{L^{\infty}}\le C_1$, must also satisfy that
\begin{equation*}
tr_{\chi} \omega_{\varphi_{\epsilon,t}} \le Ce^{A(\varphi_{\epsilon,t}- \inf_X \varphi_{\epsilon,t})}.
\end{equation*}
\end{prop}
\begin{proof}
As is explained in the beginning of the proof of the Lemma \ref{l2}, we can just do the estimate for $(*_{\epsilon,1})$ and denote $\varphi_{\epsilon,1}$ as $\varphi_{\epsilon}$ just for simplicity. Denote $h^{k \bar l}= g_{\varphi_{\epsilon}}^{k \bar j} g_{\varphi_{\epsilon}}^{i \bar l} \chi_{i \bar j}$.
Denote the operator $\widetilde{\Delta}$ as
\begin{equation*}
\widetilde{\Delta} u = h^{k \bar l} u_{k \bar l}.
\end{equation*}
Then we calculate in a normal coordinate for $\chi$ as
\begin{equation}\label{delta trchi}
\begin{split}
&\widetilde{\Delta} tr_{\chi} \omega_{\varphi_{\epsilon}} = h^{k \bar l} \Big(\chi^{i \bar j} (g_{\varphi_{\epsilon}})_{i \bar j} \Big)_{k \bar l} \\
& = h^{k \bar l}\Big( R^{i \bar j}_{k \bar l} (g_{\varphi_{\epsilon}})_{i \bar j} + \chi^{i \bar j} (g_{\varphi_{\epsilon}})_{i \bar j k \bar l} \Big)
\end{split}
\end{equation}
Here $R^{i \bar j}_{k \bar l}$ is the curvature for $\chi$. Note that  $tr_{\omega_{\varphi_{\epsilon}}}\chi = ne^{-\epsilon \varphi_{\epsilon}}$. Take logarithm of this formula and then take second derivatives of it. We can get that
\begin{equation}\label{2d equ}
\begin{split}
-\epsilon (\varphi_{\epsilon})_{k \bar l} &= \frac{1}{tr_{\omega_{\varphi_{\epsilon}}}\chi}(- h^{\bar a b} (g_{\varphi_{\epsilon}})_{b \bar a k \bar l} +h^{b \bar \alpha} g_{\varphi_{\epsilon}}^{\beta \bar a} (g_{\varphi_{\epsilon}})_{\beta \bar \alpha \bar l} (g_{\varphi_{\epsilon}})_{b \bar a k} + h^{\bar a \beta}(g_{\varphi_{\epsilon}})_{b \bar a k} (g_{\varphi_{\epsilon}})_{\beta \bar \alpha \bar l}g_{\varphi_{\epsilon}}^{b \bar \alpha} + g_{\varphi_{\epsilon}}^{i \bar j}\chi_{i \bar j k \bar l}) \\
& -\frac{(tr_{\omega_{\varphi_{\epsilon}}}\chi)_k (tr_{\omega_{\varphi_{\epsilon}}}\chi)_{\bar l}}{(tr_{\omega_{\varphi_{\epsilon}}}\chi)^2}
\end{split}
\end{equation}
Combining (\ref{delta trchi}) and (\ref{2d equ}), we can get that:
\begin{equation*}
\begin{split}
    \widetilde{\Delta}(tr_{\chi}\omega_{\varphi_{\epsilon}}) &= h^{k \bar l} R^{i \bar j}_{k \bar l} (g_{\varphi_{\epsilon}})_{i \bar j} + \chi^{k \bar l} h^{b \bar \alpha} g_{\varphi_{\epsilon}}^{\beta \bar a} (g_{\varphi_{\epsilon}})_{\beta \bar \alpha \bar l} (g_{\varphi_{\epsilon}})_{b \bar a k}+ \chi^{k \bar l}h^{\bar a \beta} (g_{\varphi_{\epsilon}})_{b \bar a k}(g_{\varphi_{\epsilon}})_{\beta \bar \alpha \bar l}g_{\varphi_{\epsilon}}^{b \bar \alpha} - g_{\varphi_{\epsilon}}^{i \bar j} R_{i \bar j}  \\
&-\frac{\chi^{k \bar l}(tr_{\omega_{\varphi_{\epsilon}}}\chi)_k (tr_{\omega_{\varphi_{\epsilon}}}\chi)_{\bar k}}{tr_{\omega_{\varphi_{\epsilon}}}\chi}+\epsilon (\varphi_{\epsilon})_{k \bar l} \chi^{k \bar l}tr_{\omega_{\varphi_{\epsilon}}}\chi \\
& \ge  h^{k \bar l} R^{i \bar j}_{k \bar l} (g_{\varphi_{\epsilon}})_{i \bar j} + \chi^{k \bar l} h^{b \bar \alpha} g_{\varphi_{\epsilon}}^{\beta \bar a} (g_{\varphi_{\epsilon}})_{\beta \bar \alpha \bar l} (g_{\varphi_{\epsilon}})_{b \bar a k} - g_{\varphi_{\epsilon}}^{i \bar j} R_{i \bar j} +\epsilon (\varphi_{\epsilon})_{k \bar l} \chi^{k \bar l}tr_{\omega_{\varphi_{\epsilon}}}\chi
\end{split}
\end{equation*}
Here in the last line above, we use the Lemma \ref{lem nable tr2}. Using the Lemma \ref{lem nable tr}, we can get that 
\begin{equation*}
\begin{split}
&\widetilde{\Delta} \log (tr_{\chi}\omega_{\varphi_{\epsilon}}) = \frac{\widetilde{\Delta} (tr_{\chi}\omega_{\varphi_{\epsilon}})}{tr_{\chi} \omega_{\varphi_{\epsilon}}}-\frac{|\widetilde{\nabla} (tr_{\chi}\omega_{\varphi_{\epsilon}})|^2}{(tr_{\chi} \omega_{\varphi_{\epsilon}})^2} \\
& = \frac{1}{tr_{\chi}\omega_{\varphi_{\epsilon}}}\Big(  h^{k \bar l} R^{i \bar j}_{k \bar l} (g_{\varphi_{\epsilon}})_{i \bar j} + \chi^{k \bar l} h^{b \bar \alpha} g_{\varphi_{\epsilon}}^{\beta \bar a} (g_{\varphi_{\epsilon}})_{\beta \bar \alpha \bar l} (g_{\varphi_{\epsilon}})_{b \bar a k} - g_{\varphi_{\epsilon}}^{i \bar j} R_{i \bar j}+\epsilon (\varphi_{\epsilon})_{k \bar l} \chi^{k \bar l}tr_{\omega_{\varphi_{\epsilon}}}\chi -\frac{|\widetilde{\nabla} (tr_{\chi}\omega_{\varphi_{\epsilon}})|^2}{tr_{\chi} \omega_{\varphi_{\epsilon}}} \Big) \\
& \ge \frac{1}{tr_{\chi}\omega_{\varphi_{\epsilon}}} \Big(h^{k \bar l} R^{i \bar j}_{k \bar l} (g_{\varphi_{\epsilon}})_{i \bar j} - g_{\varphi_{\epsilon}}^{i \bar j} R_{i \bar j} +\epsilon (\varphi_{\epsilon})_{k \bar l} \chi^{k \bar l}tr_{\omega_{\varphi_{\epsilon}}}\chi \Big).
\end{split}
\end{equation*}
Thus we have that
\begin{equation*}
\widetilde{\Delta} (\log (tr_{\chi}\omega_{\varphi_{\epsilon}})- A\varphi_{\epsilon}) \ge  \frac{1}{tr_{\chi}\omega_{\varphi_{\epsilon}}} \Big(h^{k \bar l} R^{i \bar j}_{k \bar l} (g_{\varphi_{\epsilon}})_{i \bar j} - g_{\varphi_{\epsilon}}^{i \bar j} R_{i \bar j} +\epsilon tr_{\chi} \omega_{\varphi_{\epsilon}} tr_{\omega_{\varphi_{\epsilon}}}\chi -\epsilon tr_{\chi}\omega tr_{\omega_{\varphi_{\epsilon}}}\chi \Big) -A (ne^{-\epsilon \varphi_{\epsilon}}- h^{i \bar j}g_{i \bar j}).
\end{equation*}
Since we know that $tr_{\chi}\omega_{\varphi_{\epsilon}}$ is qualitatively bounded and $\varphi_{\epsilon}$ is bounded, we can use Yau's generalized maximum principle to get that there exists a sequence $\{x_k\}$ such that 
\begin{equation*}
\lim_{k \rightarrow \infty} \log (tr_{\chi}\omega_{\varphi_{\epsilon}})- A\varphi_{\epsilon} (x_k) = \sup_{X\setminus D} \log (tr_{\chi}\omega_{\varphi_{\epsilon}})- A\varphi_{\epsilon}
\end{equation*}
and
\begin{equation*}
\limsup_{k \rightarrow \infty} \widetilde{\Delta} (\log (tr_{\chi}\omega_{\varphi_{\epsilon}})- A \varphi_{\epsilon})(x_k) \le 0.
\end{equation*}
Since $tr_{\omega_{\varphi_{\epsilon}}}\chi =ne^{-\epsilon \varphi_{\epsilon}}$ and we assume that $||\varphi_{\epsilon_{\epsilon}}||\le C_1$ for some uniform constant $C_1$, we have that
\begin{equation*}
\omega_{\varphi_{\epsilon}} \ge \frac{\chi}{C}.
\end{equation*}
Then we have that
\begin{equation*}
|\frac{1}{tr_{\chi}\omega_{\varphi_{\epsilon}}} \Big( h^{k \bar l} R^{i \bar j}_{k \bar l} (g_{\varphi_{\epsilon}})_{i \bar j} - g_{\varphi_{\epsilon}}^{i \bar j} R_{i \bar j} +\epsilon tr_{\chi} \omega_{\varphi_{\epsilon}} tr_{\omega_{\varphi_{\epsilon}}}\chi -\epsilon tr_{\chi}\omega tr_{\omega_{\varphi_{\epsilon}}}\chi \Big)| \le C.
\end{equation*}
Then, for $k$ big enough, we have that
\begin{equation}\label{e hg}
h^{i \bar j} g_{i \bar j} (x_k)-n e^{-\epsilon \varphi_{\epsilon}}(x_k) \le \frac{2C}{A}.
\end{equation}
We can choose a normal coordinate around $x_k$ for $\omega$  such that $(g_{\varphi_{\epsilon}})_{i \bar j} (x_k)= \lambda_j \delta_{ij}$ and denote $\chi_{i \bar i} (x_k)= \mu_i$.
Since $\omega$ satisfies the subsolution condition
\begin{equation*}
n \omega^{n-1} \ge (1+\delta)(n-1)\omega^{n-2}\wedge \chi,
\end{equation*}
we have that
\begin{equation}\label{e upper mu}
\frac{n}{1+\delta} \ge \Sigma_{k \neq i}\mu_k
\end{equation}
for any $i$. Since we assume that $||\varphi_{\epsilon}||_{L^{\infty}}\le C_1$ for a constant $C_1$ independent of $\epsilon,$ we can let $\epsilon$ be small enough such that 
\begin{equation}\label{e lower e}
    n(e^{-\epsilon C_1}-1) \ge -\frac{\delta n}{4(1+\delta)}.
\end{equation}

Then, we can calculate that
\begin{equation*}
\begin{split}
& \frac{2C}{A}  \ge \Sigma_k \frac{\mu_k}{\lambda_k^2} - n e^{- \epsilon \varphi_{\epsilon}} = \Sigma_k \frac{\mu_k}{\lambda_k^2} -2 \frac{\mu_k}{\lambda_k}+n e^{- \epsilon \varphi_{\epsilon}} \\
& \ge  \Sigma_{k\neq i} \mu_k (\frac{1}{\lambda_k}-1)^2 - \Sigma_{k\neq i} \mu_k + \frac{\mu_i}{\lambda_i^2} -2\frac{\mu_i}{\lambda_i} +n -\frac{\delta n}{4 (1+\delta)}\\
& \ge -\Sigma_{k\neq i} \mu_k + \frac{\mu_i}{\lambda_i^2} -2\frac{\mu_i}{\lambda_i}  +n -\frac{\delta n}{4 (1+\delta)} \ge -\frac{n}{1+\delta} + \frac{\mu_i}{\lambda_i^2} -2\frac{\mu_i}{\lambda_i}   +n -\frac{\delta n}{4 (1+\delta)} \\
&\ge \frac{\delta n}{1+\delta} -\frac{2\mu_i}{\lambda_i} -\frac{\delta n}{4 (1+\delta)}
\end{split}
\end{equation*}

In the first line of the above calculation, we use (\ref{e hg}) and $n e^{- \epsilon \varphi_{\epsilon}}= tr_{\omega_{\varphi_{\epsilon}}}\chi =\Sigma_k \frac{\mu_k}{\lambda_k}$. In the second line, we use (\ref{e lower e}). In the third line, we use (\ref{e upper mu}).
Then we can let $A$ be big such that $\frac{2C}{A}\le \frac{\delta n}{4(1+\delta)}$. Thus we get that
\begin{equation*}
\frac{\mu_i}{\lambda_i} \ge C>0
\end{equation*}
for some constant $C$. Then we have that 
\begin{equation*}
\log (tr_{\chi}\omega_{\varphi_{\epsilon}}) -A \varphi_{\epsilon} \le \sup (\log (tr_{\chi}\omega_{\varphi_{\epsilon}}) -A \varphi_{\epsilon} )= \limsup_{k \rightarrow \infty} (\log (tr_{\chi}\omega_{\varphi_{\epsilon}}) -A \varphi_{\epsilon} ) (x_k) \le C_1 -A \inf \varphi_{\epsilon}.
\end{equation*}
This concludes the proof of the $C^{1,1}$ estimate.
\end{proof}

In the proof above, we use the following Lemma which  is the Lemma 3.2 in \cite{W}:
\begin{lem}\label{lem nable tr}
Denote $h^{k \bar l}= g_{\varphi_{\epsilon}}^{k \bar j} g_{\varphi_{\epsilon}}^{i \bar l} \chi_{i \bar j}$. Then
\begin{equation*}
|\widetilde{\nabla}(tr_{\chi}\omega_{\varphi_{\epsilon}})|^2 \le tr_{\chi}\omega_{\varphi_{\epsilon}} \chi^{i \bar j} h^{r \bar q} g_{\varphi_{\epsilon}}^{p\bar s}  (g_{\varphi_{\epsilon}})_{r\bar s i} (g_{\varphi_{\epsilon}})_{p \bar q \bar j}
\end{equation*}
\end{lem}

We prove the following Lemma which deals with additional terms that appear for our new PDEs:
\begin{lem}\label{lem nable tr2}
    Denote $h^{k \bar l}= g_{\varphi_{\epsilon}}^{k \bar j} g_{\varphi_{\epsilon}}^{i \bar l} \chi_{i \bar j}$. Then
    \begin{equation*}
       \frac{h^{\bar a b} (g_{\varphi_{\epsilon}})_{\bar a b k} h^{\bar \alpha \beta} (g_{\varphi_{\epsilon}})_{\bar \alpha \beta \bar l} \chi^{k \bar l}}{tr_{\omega_{\varphi_{\epsilon}}}\chi} \le  h^{\bar a \beta} g_{\varphi_{\epsilon}}^{b \bar \alpha} \chi^{k \bar l} (g_{\varphi_{\epsilon}})_{b \bar a k} (g_{\varphi_{\epsilon}})_{\beta \bar \alpha \bar l} .
    \end{equation*}
\end{lem}
\begin{proof}
    Take a normal coordinate for $\chi$ in which $(g_{\varphi_{\epsilon}})_{i \bar j}= \lambda_i \delta_{ij}$. Then we have that $h^{i \bar j}=\frac{1}{\lambda_i^2}\delta_{ij}$. Then we can calculate that
    \begin{equation*}
        \begin{split}
        & \frac{h^{\bar a b} (g_{\varphi_{\epsilon}})_{\bar a b k} h^{\bar \alpha \beta} (g_{\varphi_{\epsilon}})_{\bar \alpha \beta \bar l} \chi^{k \bar l}}{tr_{\omega_{\varphi_{\epsilon}}}\chi} =\frac{\Sigma_{a,\beta,k}\frac{1}{\lambda_a^2}(g_{\varphi_{\epsilon}})_{\bar a a k} \frac{1}{\lambda_{\beta}^2} (g_{\varphi_{\epsilon}})_{\beta \bar \beta \bar k}}{\Sigma_i \frac{1}{\lambda_i}} \\
        &\le \frac{1}{\Sigma_i \frac{1}{\lambda_i}} \Sigma_{a, \beta} (\Sigma_k \frac{1}{\lambda_a^4}|(g_{\varphi_{\epsilon}})_{a \bar a k}|^2)^{\frac{1}{2}}(\Sigma_k \frac{1}{\lambda_{\beta}^4}|(g_{\varphi_{\epsilon}})_{\bar \beta \beta \bar k}|^2)^{\frac{1}{2}} \\
        & =\frac{1}{\Sigma_i \frac{1}{\lambda_i}} (\Sigma_a \frac{1}{\lambda_a^{\frac{1}{2}}}(\Sigma_k \lambda_a \frac{1}{\lambda_a^4}|(g_{\varphi_{\epsilon}})_{\bar a a k}|^2)^{\frac{1}{2}})^2 \le \frac{1}{\Sigma_i \frac{1}{\lambda_i}} \Sigma_a \frac{1}{\lambda_a} \Sigma_{b,k} \frac{1}{\lambda_b^3}|(g_{\varphi_{\epsilon}})_{b\bar b k}|^2 = \Sigma_{b,k} \frac{1}{\lambda_b^3}|(g_{\varphi_{\epsilon}})_{b\bar b k}|^2 \\
        &\le \Sigma_{a,b,k}\frac{1}{\lambda_a^2}\frac{1}{\lambda_b}|(g_{\varphi_{\epsilon}})_{b \bar a k}|^2 =h^{\bar a \beta} g_{\varphi_{\epsilon}}^{b \bar \alpha} \chi^{k \bar l} (g_{\varphi_{\epsilon}})_{b \bar a k} (g_{\varphi_{\epsilon}})_{\beta \bar \alpha \bar l}.
        \end{split}
    \end{equation*}
\end{proof}

After we use Evans-Krylov estimate and the standard bootstrapping in each coordinate chart, we can get the following a priori estimate:
\begin{prop}\label{higher}
Suppose there exists a subsolution $\underline{u}$ as in the Theorem \ref{main thm}. Let $C_1$ be a positive constant. Then there exists a constant $\epsilon_0$ depending on $\delta$, $\omega$, $C_1$ and $\chi$. For any $k$, there exist $C_k$ independent of $\epsilon$ such that for any $\epsilon\in (0,\epsilon_0)$ and $t\in [0,1]$, the solution $\varphi_{\epsilon,t}$ to (\ref{new path}), if satisfying $||\varphi_{\epsilon,t}||_{L^{\infty}}\le C_1$, must also satisfy that:
\begin{equation*}
||\varphi_{\epsilon,t}||_{C^k(X\setminus D)} \le C_k.
\end{equation*}
\end{prop}

This concludes the proof of the closedness of the continuity path (\ref{new path}).

\section{Proof of the Theorem \ref{main thm}}

Define $T_{\epsilon,\delta_0}=\sup\{s\in [0,1]: \forall t \in [0,s]: (*_{\epsilon,t}) \text{ has a solution } \varphi_{\epsilon,t} \in C^{\infty}(X\setminus D), \text{ with }\epsilon |\varphi_{\epsilon,t}|\le \delta_0\}$.
\begin{lem}\label{lem solve 1}
    Suppose that $\omega$ satisfies the strict subsolution condition:
    \begin{equation*}
n \omega^{n-1} \ge (1+\delta) (n-1) \omega^{n-2} \wedge \chi
\end{equation*}
and
\begin{equation*}
\frac{\omega^{n-1}\wedge \chi}{\omega^{n}}= 1 +O(\rho^{-\eta}).
\end{equation*}
for some $\eta >0$. Let $\delta_0>0$ be small enough depending on $\delta$. There there exists a constant $\epsilon_0$ such that for any $\epsilon\in (0,\epsilon_0)$, we have that $T_{\epsilon,\delta_0}=1$. Moreover, $(*_{\epsilon,1})$ admits a solution.
\end{lem}
\begin{proof}
  Since $T_{\epsilon,\delta_0}$ is connected, it suffices to prove that $T_{\epsilon,\delta_0}$ is nonempty, closed and open.

  Since $0 \in T_{\epsilon,\delta_0}$, we have that $T_{\epsilon,\delta_0}$ is nonempty.
  
  For closedness, if there exists a sequence $\{t_k\}_{k=1}^{\infty}$ in $T_{\epsilon,\delta_0}$, which converge to $t^*$, then we have that $\epsilon |\varphi_{\epsilon,t_k}|\le \delta_0$. By the Proposition \ref{prop improve linfty}, we have that there exists a constant $C$ such that 
  \begin{equation*}
      |\varphi_{\epsilon,t_k}|\le C.
  \end{equation*}
  for a uniform constant depending on $\delta_0$. Then the $C^{1,1}$ estimate follows from the Proposition \ref{c11}. Estimates of higher order derivatives of $\varphi_{\epsilon,t_k}$ follows from the Proposition \ref{higher}. Thus we can take a subsequence of $\varphi_{\epsilon,t_k}$ converging to $\varphi_{\epsilon,t^*}\in C^{\infty}(X\setminus D)$ which solves $(*_{\epsilon,t^*})$. 

  For openness, if $t^* \in T_{\epsilon,\delta_0}$, then there exists solution $\varphi_{\epsilon,t^*}\in C^{\infty}(X\setminus D)$ to $(*_{\epsilon,t^*})$ with $\epsilon|\varphi_{\epsilon,t}|\le \delta_0$. By the Proposition \ref{prop improve linfty}, we have that $|\varphi_{\epsilon,t_k}|\le C$. If $\epsilon \le \frac{\delta_0}{2C}$, we have that $\epsilon |\varphi_{\epsilon,t_k}|\le \frac{\delta_0}{2}$. Using the Proposition \ref{prop open}, there exists $\delta_1>0$ such that for any $t \in (t^*-\delta_1, t^* + \delta_1)$, we can get a solution $\varphi_{\epsilon,t}\in C^{\infty}(X\setminus D)$ to $(*_{\epsilon, t})$ with $\epsilon |\varphi_{\epsilon,t_k}|\le \delta_0$. Thus we have that $(t^* -\delta_1, t^* + \delta_1)\subset T_{\epsilon,\delta_0}$.
\end{proof}

\begin{proof}
    (of the Theorem \ref{main thm}). Without loss of generality, we use the normalization (\ref{normalize}). 
    According to the Lemma \ref{lem solve 1}, for any $(\epsilon,t)\in (0,\epsilon_0)\times [0,1]$, there exists a solution $\varphi_{\epsilon,t}\in C^{\infty}(X\setminus D)$ to $(*_{\epsilon,t})$ with $\epsilon |\varphi_{\epsilon,t}|\le \delta_0$. Using the Proposition \ref{prop improve linfty}, we can get that 
    \begin{equation*}
        |\varphi_{\epsilon,t}|\le C
    \end{equation*}
    for a uniform constant $C$. Then by the Proposition \ref{higher}, we can get uniform estimate of the $C^k$ norm of $\varphi_{\epsilon,1}$ for any $k$. Then, we can take a subsequence of $\varphi_{\epsilon_i,1}$ as $\epsilon_i$ goes to zero, such that $\varphi_{\epsilon_i,1}$ converge to some function $\varphi$ which solves $(*_{0,1})$. On the other hand, if we have a Poincar\'e type metric $\omega_{\omega_{\varphi}}$ solving $(*_{0,1})$, then $\omega_{\varphi}$ satisfies the subsolution condition
    \begin{equation*}
        n \omega_{\varphi}^{n-1}\ge (1+\delta) (n-1) \omega_{\varphi_{\epsilon}}^{n-2}\wedge \chi
    \end{equation*}
    for some positive constant $\delta$ and the asymptotic behaviour 
    \begin{equation*}
        \frac{\omega_{\varphi}^n}{\omega_{\varphi}^{n-1}\wedge \chi} =1 +O(\rho^{-\eta}).
    \end{equation*}
    Next we prove the uniqueness of solution to Poincar\'e type J-equation. Suppose that there exist two Poincar\'e type metrics $\omega_{\varphi_{i}}$, $i=1,2$, solving the J-equation:
    \begin{equation*}
        \omega_{\varphi_i}^n = \omega_{\varphi_i}^{n-1}\wedge \chi.
    \end{equation*}
    Then we can compute that
    \begin{equation}\label{e uni1}
        \begin{split}
            &0=\int_{X\setminus D} \Big( (\omega_{\varphi_1}^n - \omega_{\varphi_2}^n)- (\omega_{\varphi_1}^{n-1}\wedge \chi - \omega_{\varphi_2}^{n-1}\wedge \chi) \Big) (\varphi_1 -\varphi_2) \\
            &= \int_{X\setminus D} \int_0^1 dt dd^c(\varphi_1- \varphi_2) \wedge  \Big(n \omega_{t(\varphi_1-\varphi_2)+\varphi_2}^{n-1}-(n-1)\omega_{t(\varphi_1- \varphi_2)+ \varphi_2}^{n-2}\wedge \chi \Big)  (\varphi_1- \varphi_2) \\
            & = - \int_{X\setminus D} \int_0^1 dt d(\varphi_1- \varphi_2) \wedge d^c(\varphi_1- \varphi_2) \wedge  \Big(n \omega_{t(\varphi_1-\varphi_2)+\varphi_2}^{n-1}-(n-1)\omega_{t(\varphi_1- \varphi_2)+ \varphi_2}^{n-2}\wedge \chi \Big) 
        \end{split}
    \end{equation}
    As in the proof of the Lemma \ref{l2.1}, we have that:
    \begin{equation}\label{e uni2}
    \begin{split}
       & \frac{\sqrt{-1} dz^i \wedge d \bar z^i \wedge n \omega_{t(\varphi_1- \varphi_2)+ \varphi_2}^{n-1}}{\sqrt{-1} dz^i \wedge d \bar z^i \wedge (n-1)\omega_{t(\varphi_1- \varphi_2)+ \varphi_2}^{n-2}\wedge \chi}  \\
       & \ge e^{t \log (\frac{\sqrt{-1}dz^i \wedge d \bar z^i \wedge n \omega_{\varphi_1}^{n-1}}{\sqrt{-1}dz^i \wedge d \bar z^i \wedge (n-1)\omega_{\varphi_1}^{n-2}\wedge \chi})+ (1-t)\log (\frac{\sqrt{-1} dz^i \wedge d \bar z^i \wedge \omega_{\varphi_2}^{n-1}}{\sqrt{-1} dz^i \wedge d \bar z^i \wedge (n-1)\omega_{\varphi_2}^{n-2}\wedge \chi})} \\
       &  \ge e^{t\log (1+\delta) +(1-t)\log (1+\delta)} = 1+\delta,
    \end{split}
\end{equation}
where $\delta>0$ is a constant such that $n \omega_{\varphi_i}^{n-1} \ge (1+\delta)(n-1) \omega_{\varphi_i}^{n-2}\wedge \chi$ for $i=1,2$. Combining (\ref{e uni1}) and (\ref{e uni2}), we get that
\begin{equation*}
    \int_{X\setminus D} \int_0^1 dt d(\varphi_1- \varphi_2) \wedge d^c(\varphi_1- \varphi_2) \wedge  \delta (n-1)\omega_{t(\varphi_1- \varphi_2)+ \varphi_2}^{n-2}\wedge \chi.
\end{equation*}
This implies that $\varphi_1 = \varphi_2+C$ for some constant $C$. Thus we have $\omega_{\varphi_1}=\omega_{\varphi_2}.$
\end{proof}

\section{Poincar\'e type J-equation on K\"ahler surfaces}

In this section we prove the Theorem \ref{prop 2d} and the Corollary \ref{negative self}.

\begin{proof}
(of the Theorem \ref{prop 2d}.) Without loss of generality, we assume that $D$ only has one smooth irreducible component. Let $S\in \mathcal{O}[D]$ be a defining section of $D$. Consider
\begin{equation*}
\begin{split}
     \chi &= \chi_0 + dd^c(\eta_0 p^* \varphi_D') -a\, dd^c  \log (\lambda-\log |S|^2)+ dd^c (h) \in \mathcal{PM}_{[\chi_0]} \\
     \omega &= \omega_0 -b\, dd^c  \log (\lambda- \log |S|^2).
\end{split}
\end{equation*}
Here  $0\le \eta_0 \le 1$ is a cut-off function which is equal to $1$ when $\log (-\log |S|^2)\ge 2N$ and is zero for $\log (-\log |S|^2)\le N$, satisfying that $|\nabla_{\omega} \eta_0|_{\omega} \le \frac{C}{N}$ and $|\nabla^2_{\omega} \eta_0|_{\omega}\le \frac{C}{N}$. (See for example Lemma 4.3 in \cite{XZ}). $\varphi_D'$ is a function satisfying that $\chi_0 |_D + dd^c  \varphi_D'$ is a K\"ahler metric on $D$. $a,b,\lambda$ are positive constants. We will let $\lambda$ be big and then choose $\eta_0$ and then choose $h$ such that $dd^c h$ be small later on. Note that if we just want to prove the Theorem \ref{prop 2d}, we can just take $\varphi_D'=0$ and $h=0$. We need to consider such general metric in the proof of the Corollary \ref{thm lower k} below. Since $\omega_0$ and $\chi_0$ are smooth K\"ahler metrics on $X$, we can write 
    \begin{equation*}
        \chi_0=p^* \chi_D + O(\frac{1}{-\log|S|^2}),  \omega_0=p^* \omega_D + O(\frac{1}{-\log|S|^2})
    \end{equation*}
    for some K\"ahler metrics $\chi_D$ and $\omega_D$ on $D$. Here $O(\frac{1}{-\log|S|^2})$ is with respect to a Poincar\'e type metric. In a cusp coordinate where $D=\{z_2=0\}$, we can write $|S|^2=|z_2|^2 e^F$ for some smooth function $F$. Then we can compute
\begin{equation*}
    \begin{split}
        & ( - \log (\lambda-\log |S|^2))_{1\bar 1} = \frac{F_{1\bar 1}}{\lambda- \log |z_2|^2 -F} +\frac{F_1 F_{\bar 1}}{(\lambda- \log |z_2|^2-F)^2},\\
        &( - \log (\lambda-\log |S|^2))_{1\bar 2}= \frac{F_{1 \bar 2}}{\lambda-\log |z_2|^2 -F} +\frac{F_1 (1+\bar z_2 F_{\bar 2})}{\bar z_2 (\lambda- \log |z_2|^2 - F)^2} ,\\
        &( - \log (\lambda-\log |S|^2))_{2\bar 1}=\frac{F_{\bar 1  2}}{\lambda-\log |z_2|^2 -F} + \frac{F_{\bar 1} (1+ z_2 F_{ 2})}{ z_2 (\lambda- \log |z_2|^2 - F)^2},\\
        & ( - \log (\lambda-\log |S|^2))_{2\bar 2}= \frac{F_{2\bar 2}}{\lambda- \log |z_2|^2 -F} + \frac{1}{|z_2|^2 \log^2 |z_2|^2} \frac{(-\log |z_2|^2)^2}{(\lambda- \log |z_2|^2 -F)^2}(1+z_2 F_2)(1+\bar z_2 F_{\bar 2}).
    \end{split}
\end{equation*}
We compute that
\begin{equation}\label{e 2d asy}
\begin{split}
     &\frac{\omega^2}{C_{[\omega_0],[\chi_0]} \chi \wedge \omega} \\
     &= \frac{2 \omega_D \wedge \frac{b}{|z_2|^2 (\lambda- \log |z_2|^2 -F)^2} \sqrt{-1}dz^2 \wedge d \bar z^2}{C_{[\omega_0],[\chi_0]}( (\chi_D +dd^c \varphi_D')\wedge \frac{b}{|z_2|^2 (\lambda- \log |z_2|^2 -F)^2} \sqrt{-1}dz^2 \wedge d \bar z^2 + \omega_D \wedge \frac{a}{|z_2|^2 (\lambda- \log |z_2|^2 -F)^2} \sqrt{-1}dz^2 \wedge d \bar z^2)} \\
     &+ O(\frac{1}{(\lambda- \log |z_2|-F)})\\
     &= \frac{2}{C_{[\omega_0],[\chi_0]}(\frac{\chi_D +dd^c \varphi_D'}{\omega_D} +\frac{a}{b})} + O(\frac{1}{(\lambda- \log |z_2|-F)})
\end{split}
\end{equation}
 Since $D$ is a Riemann surface, there exists a constant $C_D$ such that $[\omega_D]= C_D [\chi_D]$. So there exists a smooth function $\varphi_D$ on $D$ such that $\omega_D + dd^c \varphi_D = C_D (\chi_D+ dd^c \varphi_D')$. Since $[2\omega_0 - C_{[\omega_0],[\chi_0]} \chi_0]$ is a K\"ahler class on $X$, its restriction to $D$, which is $[2\omega_D - C_{[\omega_0],[\chi_0]} \chi_D]$ is a K\"ahler class on $D$. This implies that 
 \begin{equation*}
     2C_D > C_{[\omega_0],[\chi_0]}.
 \end{equation*}
 Define 
 \begin{equation*}
     \omega'= \omega + dd^c (\eta_0 p^* \varphi_D)= \omega+ \eta_0 dd^c p^* \varphi_D + d \eta_0 \wedge d^c p^* \varphi_D + d p^* \varphi_D \wedge d^c \eta_0 + p^* \varphi_D dd^c \eta_0,
 \end{equation*}

Next, we want to show that $\omega'$ can be used as a background metric satisfying (\ref{e sub0}) and (\ref{sub asym}) in the Theorem \ref{main thm}. First, since $2C_D >C_{[\omega_0],[\chi_0]}$, we can select $b>0$ such that 
\begin{equation}\label{e cxcd}
    \frac{2}{C_{[\omega_0],[\chi_0]}(\frac{1}{C_D}+\frac{a}{b})}=1.
\end{equation}
Then the definition of $\varphi_D$ and (\ref{e 2d asy}) imply that 
\begin{equation*}
    \frac{\omega'^2 }{C_{[\omega_0],[\chi_0]} \chi \wedge \omega'}=1+ O(\frac{1}{-\log |S|^2}).
\end{equation*}
which implies (\ref{sub asym}). We compute
\begin{equation*}
\begin{split}
    2\omega'- C_{[\omega_0],[\chi_0]} \chi & = \eta_0( 2 \omega_0 +2 dd^c p^* \varphi_D - C_{[\omega_0],[\chi_0]} (\chi_0+ dd^c \varphi_D') ) +(1-\eta_0)(2\omega_0 - C_{[\omega_0],[\chi_0]} \chi_0) \\
    &+2 d \eta_0 \wedge d^c p^* \varphi_D +2 d p^* \varphi_D \wedge d^c \eta_0 +2p^* \varphi_D dd^c \eta_0 \\
    & - C_{[\omega_0],[\chi_0]} d\eta_0 \wedge d^c p^* \varphi_D'-C_{[\omega_0],[\chi_0]} d p^* \varphi_D' \wedge d^c \eta_0 - C_{[\omega_0],[\chi_0]} p^* \varphi_D' dd^c \eta_0 - C_{[\omega_0],[\chi_0]} dd^c h\\
    &-(2b-C_{[\omega_0],[\chi_0]} a) dd^c \log (\lambda -\log |S|^2).
\end{split}
\end{equation*}
 Since $2\omega_0 -C_{[\omega_0],[\chi_0]} \chi_0>0$ and 
 \begin{equation*}
 2\omega_D + 2dd^c \varphi_D  - C_{[\omega_0],[\chi_0]} (\chi_D+ dd^c \varphi_D')  = (2C_D- C_{[\omega_0],[\chi_0]})(\chi_D+ dd^c \varphi_D') >0,    
 \end{equation*}
 we can get that 
 \begin{equation*}
     \eta_0( 2 \omega_0 +2 dd^c p^* \varphi_D - C_{[\omega_0],[\chi_0]} \chi_0 ) +(1-\eta_0)(2\omega_0 - C_{[\omega_0],[\chi_0]} \chi_0) \ge \epsilon \omega_0- C \eta_0 \sqrt{-1} dz^2 \wedge d \bar z^2
 \end{equation*}
for some positive constant $\epsilon$ as long as $N \ge N_0$ for some constant $N_0$ depending only on $\omega_0$ and $\chi_0$. According to (\ref{e cxcd}), we have that $2b >aC_{[\omega_0],[\chi_0]}$ which implies that
\begin{equation*}
    \epsilon \omega_0  -(2b-C_{[\omega_0],[\chi_0]} a) dd^c \log (\lambda -\log |S|^2) \ge \epsilon_1 \chi
\end{equation*}
is a Poincar\'e type K\"ahler metric if we let $\lambda$ be big. From now on, we fix $\lambda$. Then we let $N$ be big and then let $dd^c h$ be small such that 
\begin{equation*}
\begin{split}
& 2\omega'- C_{[\omega_0],[\chi_0]} \chi \\
&\ge  \epsilon \omega_0 - C \eta_0 \sqrt{-1} dz^2 \wedge d \bar z^2 + 2 d \eta_0 \wedge d^c p^* \varphi_D +2 d p^* \varphi_D \wedge d^c \eta_0 +2p^* \varphi_D dd^c \eta_0\\
    & - C_{[\omega_0],[\chi_0]} d\eta_0 \wedge d^c p^* \varphi_D'-C_{[\omega_0],[\chi_0]} d p^* \varphi_D' \wedge d^c \eta_0 - C_{[\omega_0],[\chi_0]} p^* \varphi_D' dd^c \eta_0 - C_{[\omega_0],[\chi_0]} dd^c h\\
&-(2b-C_{[\omega_0],[\chi_0]} a) dd^c \log (\lambda -\log |S|^2) \\
& \ge \epsilon_1 \chi - C \eta_0 \sqrt{-1} dz^2 \wedge d \bar z^2 + 2 d \eta_0 \wedge d^c p^* \varphi_D +2 d p^* \varphi_D \wedge d^c \eta_0 +2p^* \varphi_D dd^c \eta_0 \\
& - C_{[\omega_0],[\chi_0]} d\eta_0 \wedge d^c p^* \varphi_D'-C_{[\omega_0],[\chi_0]} d p^* \varphi_D' \wedge d^c \eta_0 - C_{[\omega_0],[\chi_0]} p^* \varphi_D' dd^c \eta_0 - C_{[\omega_0],[\chi_0]} dd^c h\\
& \ge \epsilon_2 \chi
\end{split}
\end{equation*}
is a Poincar\'e type metric. Thus we can solve the Poincar\'e type J-equation by the Theorem \ref{main thm} using $\omega'$ as a background metric.
\end{proof}

\begin{proof}
  (of the Corollary \ref{negative self}.) We compute that
  \begin{equation*}
      [2\omega -C_{[\omega_0],[\chi_0]} K_X[D]]^2 = 4[\omega]^2 -4 C_{[\omega_0],[\chi_0]} [\omega]\cdot K_X[D] + C_{[\omega_0],[\chi_0]}^2 K_X[D]^2 = C_{[\omega_0],[\chi_0]}^2 K_X[D]^2 >0.
  \end{equation*}
  and
  \begin{equation*}
      [2\omega -C_{[\omega_0],[\chi_0]} K_X[D]]\cdot [\omega]= 2[\omega]^2 -\frac{[\omega]^2}{K_X[D] \cdot [\omega]} K_X[D] \cdot [\omega] =[\omega]^2 >0
  \end{equation*}
  Since there is no curve of negative self-intersections, By Proposition 4.5 in \cite{SW} which is first proved in \cite{L}, $[2\omega]- C_{[\omega_0],[\chi_0]} K_X[D]$ is a K\"ahler class. Thus we can use the Theorem \ref{prop 2d} to prove this Corollary.
\end{proof}

\begin{lem}\label{lem K-energy}
    Suppose that $K_X[D]$ is ample. Suppose that there exist Poincar\'e type metrics $\omega$ and $\chi$ such that $- Ric  (\omega) = \chi$.  Then if Poincar\'e type J-equation is solvable:
    \begin{equation*}
        \omega_{\varphi}^n = C \omega_{\varphi}^{n-1}\wedge \chi,
    \end{equation*}
    then we have that the K-energy is bounded from below in the space of K\"ahler metrics $\mathcal{PM}_{[\omega]}$.
\end{lem}
\begin{proof}
    Using the Lemma 4.1 in \cite{XZ}, we can get the expression of the Hessian of J-functional along Poincar\'e type $C^{1,1}$ geodesic:
    \begin{equation*}
    \begin{split}
        &d_R d^c_R (\frac{\overline{R}}{n+1}\mathcal{E}(\varphi)-\mathcal{E}^{Ric(\omega)}(\varphi))= \int_{X\setminus D} \frac{\overline{R}}{n+1}(\pi^* \omega+ dd^c U)^{n+1}- (\pi^* \omega+ dd^c U)^n \wedge \pi^* Ric(\omega)\\
        &= \int_{X\setminus D} - \pi^* Ric(\omega)\wedge (\pi^* \omega+ dd^c U)^n \ge 0.
    \end{split}
    \end{equation*}
    Here we see a geodesic $U(x,t)=u_t(x)$ as a $S^1$ invariant function on $(X\setminus D) \times R$, where $R=S^1 \times [0,1]$ is a cylinder. By geodesic equation, $(\pi^* \omega+ dd^c U)^{n+1}=0$. Thus we have that the functional $\frac{\overline{R}}{n+1}\mathcal{E}-\mathcal{E}^{Ric(\omega)}$ is convex along $C^{1,1}$ Poincar\'e type geodesics. Note that the Poincar\'e type J-equation is the critical point of $\frac{\overline{R}}{n+1}\mathcal{E}-\mathcal{E}^{Ric(\omega)}$, the functional $\frac{\overline{R}}{n+1}\mathcal{E}-\mathcal{E}^{Ric(\omega)}$ is bounded from below. Note that $H_{\mu_0}(\omega_{\varphi}^n)$ is bounded from below. The decomposition of K-energy (\ref{decom kenergy}) implies that K-energy is bounded from below. 
\end{proof}

\begin{lem}\label{lem negative ric}
Suppose that $ K_X[D] $ is ample. Then for any Kähler class $ \mathcal{PM}_{[\omega_1]} $, there exist Poincar\'e type Kähler metrics
$$
\chi \in \mathcal{PM}_{K_X[D]}, \qquad \omega \in \mathcal{PM}_{[\omega_1]},
$$
such that $ \chi $ is locally asymptotic to a product metric near $ D $ and
\begin{equation*}
- \mathrm{Ric}(\omega) = \chi.
\end{equation*}

Moreover, for any smooth Kähler metric $ \chi_0 \in K_X[D] $, there exists a function $ \varphi_D' $ on $ D $ such that we can choose $\chi$ to be
\begin{equation*}
\chi = \chi_0 + dd^c(\eta_0 p^* \varphi_D') - dd^c \log(\lambda - \log |S|^2) + dd^c h,
\end{equation*}
where $ \eta_0 $ is a cut-off function, $ \lambda $ is chosen so that $ \chi $ is of Poincar\'e type, and $ h $ satisfies that $ dd^c h $ is compactly supported in $ X \setminus D $ and can be made arbitrarily small once $ \eta_0 $ and $ \lambda $ are fixed.
\end{lem}
\begin{proof}
   This lemma follows from the Theorem 3.3 in \cite{A}. We provide a sketch of the proof for the reader’s convenience. First, we take a smooth K\"ahler metric $\omega_0 \in [\omega_1]$. Define $\omega = \omega_0 -b dd^c \log (\lambda-\log |S|^2)$ which is a Poincar\'e type K\"ahler metric. Denote $\omega_D = \omega_0 |_{D}$. Then $\omega$ can be written as 
   \begin{equation*}
       \omega = p^* \omega_D + b e^{-t} dt \wedge \widetilde{\eta} + O(e^{-\eta t}).
   \end{equation*}
   Since $K_X[D]$ is ample, we get that $K_D$ is ample. Since $-Ric (\omega_D) \in K_D$, we can find a smooth function $\varphi_D$ on $D$ such that $-Ric(\omega_D + dd^c \varphi_D)$ is a positive K\"ahler metric on $D$. Define 
   \begin{equation*}
       \omega_2 = \omega + dd^c (\eta_0 p^* \varphi_D) = p^* (\omega_D + dd_D^c \varphi_D) + b e^{-t} dt \wedge \widetilde{\eta} + O(e^{-\eta t}).
   \end{equation*}
       Then we have that
    \begin{equation*}
        \omega_2^n = n \Big(p^*(\omega_D + dd_D^c \varphi_D ) \Big)^{n-1}\wedge b e^{-t} dt \wedge \widetilde{\eta} + O(e^{-\eta t}).
    \end{equation*}    
    Then we can compute that
    \begin{equation*}
         Ric (\omega_2) = -dd^c \log \omega_2^n=  p^* Ric (\omega_D + dd_D^c \varphi_D) - e^{-t} dt \wedge \widetilde{\eta} + O(e^{-\eta t}) 
    \end{equation*}
    Note that although $-Ric (\omega_2)$ is a Poincar\'e type K\"ahler metric near $D$, it is possible that it may not be positive away from $D$. As a result, we can construct a Poincar\'e type K\"ahler metric $\chi_1 \in \mathcal{PM}_{K_X[D]}$ in the form
    \begin{equation*}
      \chi_1=  \chi_0 + dd^c (\eta_0 p^* \varphi_D') -dd^c \log (\lambda- \log |S|^2)
    \end{equation*}
    with $\varphi_D'$ to be a function defined on $D$ such that $\chi_0 |_D + dd_D^c \varphi_D' =- Ric (\omega_D + dd_D^c \varphi_D)$.
    Then by Auvray's $\partial \bar \partial-$lemma (Proposition 3.5 in \cite{A}) for Poincar\'e type K\"ahler metric, we can find $f \in \mathbb{R} \oplus C_{-1}^{\infty}(X\setminus D)$ such that 
    \begin{equation*}
        \chi_1= -Ric (\omega_2) + dd^c f.
    \end{equation*}
  Then following Auvray's argument in Theorem 3.3 in \cite{A}, we can construct another function $f_0 \in C_{-1}^{\infty}(X\setminus D)$ with $\int_{X\setminus D} e^{f_0} \omega_2^n =1$, such that $dd^c (f_0-f)$ can be made arbitrarily small and compactly supported in $X\setminus D$. In particular, we have 
    \begin{equation*}
       \chi \triangleq -Ric(\omega_2)+ dd^c f_0= \chi_1 + dd^c (f_0-f)= \chi_0 + dd^c (\eta_0 p^* \varphi_D')- dd^c \log (\lambda- \log |S|^2) + dd^c (f_0-f)
    \end{equation*}
    is a Poincar\'e type K\"ahler metric. Using Theorem 3.2 in \cite{A} to solve a Poincar\'e type complex Monge Amp\'ere equation:
    \begin{equation*}
        (\omega_2+ dd^c \varphi)^2 = e^{f_0} \omega_2^n,
    \end{equation*}
    we can find a Poincar\'e type K\"ahler metric $\omega_2+ dd^c \varphi \in \mathcal{PM}_{[\omega_1]}$ such that 
    \begin{equation*}
        -Ric (\omega_2+ dd^c \varphi)= \chi.
    \end{equation*}
\end{proof}

\begin{proof}
(of the Corollary \ref{thm lower k})
Using the assumption that $K_X[D]$ is ample, for any K\"ahler class $\mathcal{PM}_{[\omega]}$, we can use the Lemma \ref{lem negative ric} to construct
\begin{equation*}
\chi= \chi_0 + dd^c (\eta_0 p^* \varphi_D') -dd^c \log (\lambda- \log |S|^2) + dd^c h
\end{equation*}
such that $\chi= -Ric(\omega_1)$ for some $\omega_1 \in \mathcal{PM}_{[\omega]}$.
Here $\eta_0$ and $\lambda$ can be chosen freely as long as $\chi$ is positive and then $dd^c h$ can be made arbitrarily small. Under the assumption that $X$ is a K\"ahler surface with no curves of negative self-intersections and $K_X[D]$ is ample, we can use the proof of Theorem \ref{prop 2d} to solve the Poincar\'e type J-equation 
    \begin{equation*}
    \omega_{\varphi}^n = C \omega_{\varphi}^{n-1}\wedge \chi.
    \end{equation*}
    in $\mathcal{PM}_{[\omega]}$, as long as we let $\lambda$ be big and choose $\eta_0$ appropriately and let $dd^c h$ be small enough. Then we can use the Lemma \ref{lem K-energy} to get that K-energy is bounded from below on $\mathcal{PM}_{[\omega]}$.
\end{proof}

\section{Asymptotic behaviour of solution to J-equation}
Throughout this section, we use the normalization (\ref{normalize}). Without loss of generality, we assume that $D$ only has one irreducible component.

In this section, we use Auvray's idea in proving the asymptotic behaviour of Poincar\'e type extremal K\"ahler metrics, c.f. \cite{A2} to prove the asymptotic behaviour of the solution to J-equation which is the Theorem \ref{thm asy}

\subsection{Construction of background metric}

In this subsection, we use Auvray's idea in \cite{A3} to prove:
\begin{prop}\label{prop bdd}
    There exists a Poincar\'e type K\"ahler metric $\omega$ which is locally asymptotic to a product metric near $D$ such that the solution $\omega_{\varphi}=\omega+ dd^c \varphi$ to the J-equation satisfies that $\varphi$ is bounded.
\end{prop}

For any function $u$ supported in a small neighbourhood of $D$, we can decompose it as
\begin{equation*}
u = \Pi_0 u + u^{\bot}
\end{equation*}
as explained in the section 2.5. Since the integral of $u^{\bot}$ on each $S^1$ fiber is zero and the length of the $S^1$ fiber exponentially decay to zero as $t$ goes to $\infty$, we have the following Lemma basically saying that the decay rate of $u^{\bot}$ con be improved if we have control on its higher order derivatives. See the section 3 of \cite{A2} and the Formula (3.6) of \cite{S}.

\begin{lem}\label{improve decay}
For any $\delta \in \mathbb{R}$ and $k\in \mathbb{N}$, there exists a constant $C$ such that:
\begin{equation*}
||u^{\bot}||_{W_{\delta}^{k,2}}\le C||u^{\bot}||_{W_{\delta+1}^{k+1,2}}. 
\end{equation*}
holds for any $k$ and $u$ such that $||u^{\bot}||_{W_{\delta+1}^{k+1,2}}< \infty$.
\end{lem}

\begin{lem}\label{improve decay holder}
For any $\delta \in \mathbb{R}$ and $k\in \mathbb{N}$, there exists a constant $C$ such that:
\begin{equation*}
||u^{\bot}||_{C_{\delta}^{k,\alpha}}\le C||u^{\bot}||_{C_{\delta+1}^{k+1,\alpha}}. 
\end{equation*}
holds for any $k$ and $u$ such that $||u^{\bot}||_{C_{\delta+1}^{k+1,\alpha}}< \infty$.
\end{lem}

We also have the following technical lemma:
\begin{lem}\label{u dec}
For any $\delta \in \mathbb{R}$, there exists a uniform constant $C$ such that $$||u^{\bot}||_{W_{\delta}^{k,2}}\le C ||u||_{W^{k,2}_{\delta}}\text{ and }||u_0||_{W_{\delta}^{k,2}}\le C ||u||_{W^{k,2}_{\delta}}$$ for any $u$ such that $||u||_{W^{k,2}_{\delta}}\le +\infty$.
\end{lem}

\begin{lem}\label{lem cd}
    Suppose that a Poincar\'e type K\"ahler metric $\omega$ is a strict subsolution:
    \begin{equation*}
        n\omega^{n-1} \ge (1+\delta)(n-1)\omega^{n-2}\wedge \chi
    \end{equation*}
    for some constant $\delta>0$. Then we have that on $D$:
    \begin{equation*}
        n([\omega]|_D)^{n-1} -(n-1)([\omega]|_D)^{n-2}\cdot [\chi]|_D >0
    \end{equation*}
\end{lem}
\begin{proof}
    First, we write $\omega= \omega_0 +dd^c \varphi$ and $\chi= \chi_0 +dd^c \psi$ for some smooth K\"ahler metrics $\omega_0$ and $\chi_0$ on $X$. It suffices to prove that
    \begin{equation*}
        \int_D n \omega_0|_D^{n-1}- (n-1)\omega_0 |_{D}^{n-2}\wedge \chi_0 |_D >0.
    \end{equation*}
    Let $\varphi= \varphi_0 (t,p) + \varphi^{\bot}$ and $\psi = \psi_0(t,p)+ \psi^{\bot}$ be the decomposition in section 2.5. Using the Lemma \ref{improve decay holder}, we get that 
    \begin{equation}\label{e bot}
        \varphi^{\bot}=O(e^{-t}) ,\,\,\,\,\psi^{\bot}= O(e^{-t}).
    \end{equation}
Take a cusp coordinate $(z)$ around an arbitrary point $p\in D$ such that in this coordinate $D=\{z_n=0\}$. Write $(z)=(z',z_n)$. In this coordinate we use (\ref{e bot}) to get that 
\begin{equation*}
    \omega_0(z',0)+ dd^c_D \varphi_0(z',t) = \omega_0(z',z_n)+ O(|z_n|) + dd^c_D \varphi (z',z_n) +O(e^{-t})
\end{equation*}
and
\begin{equation*}
    \chi_0 (z',0)+ dd^c_D \psi_0 (z',t) = \chi_0 (z',z_n)+ O(|z_n|)+ dd^c_D \psi(z',z_n)+ O(e^{-t}).
\end{equation*}
Then, using the subsolution condition for $\omega$, we can compute that
\begin{equation*}
    \begin{split}
        & n (\omega_0 (z',0)|_D+ dd^c_D \varphi_0 (z',t))^{n-1} -(n-1)(\omega_0 (z',0)|_D+ dd^c_D \varphi_0 (z',t))^{n-2}\wedge (\chi_0(z',0)|_D+ dd^c_D \psi_0(z',t))\\
        & = n (\omega_0 (z',z_n) |_D+ dd^c_D \varphi (z',z_n))^{n-1}\\
        &-(n-1)(\omega_0(z',z_n) |_D+ dd^c_D \varphi (z',z_n))^{n-2}\wedge (\chi_0(z',z_n) |_D+dd^c_D \psi(z',z_n)) +O(e^{-t}) \\
        & \ge \delta (n-1) (\omega_0(z',z_n) |_D + dd^c_D \varphi (z',z_n))^{n-2}\wedge (\chi_0 (z',z_n) |_D + dd^c_D \psi (z',z_n)) +O(e^{-t}) \\
        & \ge \delta(n-1) (\omega_0(z')|_D+ dd^c_D \varphi_0 (z',t))^{n-1}\wedge (\chi_0 (z')|_D+dd^c_D \psi_0(z',t)) +O(e^{-t}).
    \end{split}
\end{equation*}
Here $\beta |_D =\Sigma_{i,j=1}^{n-1} \beta_{i \bar j}$ means taking the part of $\beta$ corresponds to $z'$.
 Then we have that 
    \begin{equation*}
    \begin{split}
          &   \int_D n \omega_0|_D^{n-1}- (n-1)\omega_0 |_{D}^{n-2}\wedge \chi_0 |_D \\
        & = \int_D n (\omega_0 |_D + dd^c_D \varphi_0 (t,\cdot))^{n-1}- (n-1)(\omega_0 |_D + dd_D^c \varphi_0 (t, \cdot))^{n-2}\wedge (\chi_0 |_D + dd^c_D \psi_0(t,\cdot) ) \\
        & \ge \int_D \delta(n-1) (\omega_0(z')+ dd^c_D \varphi_0 (z',t))^{n-1}\wedge (\chi_0 (z')+dd^c_D \psi_0(z',t)) +O(e^{-t})\\
        & \int_D \delta (n-1) \omega_0^{n-1}\wedge \chi_0 + O(e^{-t})
    \end{split}
    \end{equation*}
    which is positive as long as we let $t$ be big enough.
\end{proof}

We denote the constant $C_D= (n-1)\frac{[\omega] |_D^{n-2} \cdot [\chi] |_D}{[\omega] |_D^{n-1}}$.

\begin{proof}
 (of the Proposition \ref{prop bdd}). We use Auvray's way of constructing Poincar\'e type K\"ahler metrics, see the section 2.1, by defining 
    \begin{equation*}
        \omega= \omega_X -Ai \partial \bar \partial \log (\lambda - \log |S|^2),
    \end{equation*}
    where $S\in (\mathcal{O}([D]))$ is a defining section of $D$. Near $D$, we have that:
    \begin{equation*}
        \omega= p^*\omega_D + A dt \wedge e^{-t} \widetilde{\eta} + O(e^{-t})
    \end{equation*}
    Here $\omega_D = \omega_X|_D$. Recall that $\chi$ is locally asymptotic to a product metric near $D$:
    \begin{equation*}
        \chi = b dt \wedge 2 e^{-t} \widetilde{\eta} + p^* \chi_D +O(e^{-\eta t}).
    \end{equation*}
    First, we solve the Poisson equation on $D$ to get a smooth function $\widetilde{\psi}$ such that 
    \begin{equation*}
        tr_{\omega_D}(\chi_D+ dd^c_D \widetilde{\psi})=C_D.
    \end{equation*}
    Define $\widetilde{\chi}\triangleq \chi+ dd^c (\eta_0 p^* \widetilde{\psi})$, where $0\le \eta_0 \le 1$ is a cut-off function which is equal to $1$ in a small neighbourhood of $D$ and is supported in a slightly bigger neighbourhood of $D$, satisfying that $|\nabla_{\omega} \eta_0|_{\omega}$ and $|\nabla^2_{\omega} \eta_0|_{\omega}$ can be made very small. (See for example Lemma 4.3 in \cite{XZ})
    Then we can compute that
    \begin{equation*}
    \begin{split}
          \frac{\omega^{n-1}\wedge \widetilde{\chi}}{\omega^n} &=\frac{b dt \wedge 2 e^{-t}\widetilde{\eta} \wedge p^* \omega_D^{n-1}+ p^* (\chi_D +dd^c_D \widetilde{\psi}) \wedge (n-1) Adt \wedge e^{-t}\widetilde{\eta} \wedge p^* \omega_D^{n-2}}{n A dt \wedge e^{-t}\widetilde{\eta} \wedge p^* \omega_D^{n-1}} +O(e^{-\eta t})\\
          & =\frac{\frac{2b}{A} + tr_{\omega_D}(\chi_D+dd^c_D \widetilde{\psi})}{n} +O(e^{-\eta t}) = \frac{\frac{2b}{A}+C_D}{n} + O(e^{-\eta t}).
    \end{split}
    \end{equation*}
    By the Lemma \ref{lem cd}, $C_D <n$. Thus we can let $A$ to be $\frac{2b}{n-C_D}>0$ such that $\omega$ is positive and satisfies
    \begin{equation}\label{tildechi asy}
         \frac{\omega^{n-1}\wedge \widetilde{\chi}}{\omega^n} =1+O(e^{-\eta t}).
    \end{equation}
    Then we can compute that
    \begin{equation}\label{e lin0}
    \begin{split}
         0 &= \int_{\{t \le s\}}(\omega_{\varphi}^n -\omega_{\varphi}^{n-1}\wedge \chi)e^t  =  \int_{\{t \le s\}}(\omega_{\varphi}^n -\omega_{\varphi}^{n-1}\wedge \widetilde{\chi})e^t + \int_{\{t\le s\}} \omega_{\varphi}^{n-1}\wedge dd^c (\eta_0 p^* \widetilde{\psi})e^t \\
       & = \int_{ \{t \le s\} } (\omega^n - \omega^{n-1}\wedge \widetilde{\chi})e^t +  \int_{\{t \le s\}} \Big( (\omega_{\varphi}^n -\omega_{\varphi}^{n-1}\wedge \chi) -(\omega^n - \omega^{n-1}\wedge \widetilde{\chi}) \Big)e^t \\
       & + \int_{\{t\le s\}} \omega_{\varphi}^{n-1}\wedge dd^c (\eta_0 p^* \widetilde{\psi})e^t  \\
       & = O(1)+  \int_{\{t \le s\}} \Big( (\omega_{\varphi}^n -\omega_{\varphi}^{n-1}\wedge \widetilde{\chi}) -(\omega^n - \omega^{n-1}\wedge \widetilde{\chi}) \Big)e^t + \int_{\{t\le s\}} \omega_{\varphi}^{n-1}\wedge dd^c (\eta_0 p^* \widetilde{\psi})e^t \\
       & =O(1) + I + II
    \end{split}
    \end{equation}
    In the fourth line above, we use (\ref{tildechi asy}). 
    Recall the decomposition $\varphi=\varphi_0(t,p)+ \varphi^{\bot}$ in section 2.5. We can further decompose $\varphi_0$ as 
    \begin{equation*}
        \varphi_0 (t,p) = \varphi_{00}(t)+\varphi_{01}(t,p)
    \end{equation*}
    where $\varphi_{01}$ satisfies that
    \begin{equation*}
        \int_{D} \varphi_{01}(t,p)dp =0
    \end{equation*}
    for any $t$. Note that all the derivatives of $\varphi_{01}(t,p)$ are bounded. Then we can use the Sobolev inequality on each $D-$slice: $D\times \{t\}$ to prove that $\varphi_{01}$ is bounded. Denote $\Theta= (\Sigma_{i=0}^{n-1}\omega_{\varphi}^i \wedge \omega^{n-1-i}) -(\Sigma_{i=0}^{n-2}\omega^i \wedge \omega_{\varphi}^{n-2-i} \wedge \widetilde{\chi})$. Then we can compute that
    \begin{equation}\label{e lin2}
    \begin{split}
         &I= \int_{\{t \le s\}} \Big( (\omega_{\varphi}^n -\omega_{\varphi}^{n-1}\wedge \widetilde{\chi}) -(\omega^n - \omega^{n-1}\wedge \widetilde{\chi}) \Big)e^t = \int_{\{t\le s\}} e^t dd^c \varphi \wedge \Theta  \\
        & = \int_{\{t\le s\}} \varphi dd^c (e^t)\wedge \Theta + e^s \int_{\{t =s \}} (d^c \varphi -\varphi d^c t)\wedge \Theta \\
        & =O(1) -e^s \int_{\{t =s\}} \varphi_{00}(s) d^c t \wedge \Theta \\
        & = O(1)+ 2\pi \varphi_{00}(s) \int_D \Theta \\
        & = 2\pi \varphi_{00}(s) (n [\omega_D]^{n-1}-(n-1)[\omega_D]^{n-2}\cdot[\chi_D]) +O(1).
    \end{split}
    \end{equation}
    In the third line above, we use the fact that $dd^c(e^t), d^c\varphi, \varphi_{01}, \varphi^{\bot}$ are all bounded. We also compute:
    \begin{equation}\label{e lin3}
        \begin{split}
            II&= \int_{\{t\le s\}} \omega_{\varphi}^{n-1}\wedge dd^c (\eta_0 p^* \widetilde{\psi})e^t \\
            &= \int_{\{t \le s\}} \omega_{\varphi}^{n-1}\wedge dd^c e^t \wedge \eta_0 p^* \widetilde{\psi} +\int_{\{t=s\}}e^s (d^c (\eta_0 p^* \widetilde{\psi}) - \eta_0 p^* \widetilde{\psi} d^c t)\wedge \omega_{\varphi}^{n-1} =O(1).
        \end{split}
    \end{equation}
    Here we use the fact that $dd^c(e^t), \widetilde{\psi}, d^c \widetilde{\psi}$ are all bounded.
    
    Combining (\ref{e lin0}), (\ref{e lin2}) and (\ref{e lin3}), we get
    \begin{equation*}
        \varphi_{00}(s) (n [\omega_D]^{n-1}-(n-1)[\omega_D]^{n-2}\cdot[\chi_D]) =O(1).
    \end{equation*}
    By the Lemma \ref{lem cd}, we have that $n [\omega_D]^{n-1}-(n-1)[\omega_D]^{n-2}\cdot[\chi_D]>0$. Thus we have that $\varphi_{00}$ is uniformly bounded. Since $\varphi_{01}$ and $\varphi^{\bot}$ are also bounded which we explained before, we have that $\varphi$ is bounded. This concludes the proof of the Proposition \ref{prop bdd}.
\end{proof}

\subsection{Splitting theorem}
Recall that $\chi$ is assumed to be locally asymptotic to a product metric:
\begin{equation*}
    \chi= p^* \chi_{D} + 2b e^{-t} dt \wedge \widetilde{\eta} +O(e^{-\eta t})
\end{equation*}
for some K\"ahler metric $\chi_D$ on $D$. Let $\omega$ be the metric that we constructed in the Proposition \ref{prop bdd} which is also locally asymptotic to a product metric:
\begin{equation*}
    \omega= p^* \omega_{D} + 2a e^{-t} dt \wedge \widetilde{\eta} +O(e^{-\eta t})
\end{equation*}
for some K\"ahler metric $\omega_D$ on $D$.

In this subsection, we want to prove a splitting theorem for J-equation. Take $z$ to be the coordinate on $\Delta^*$. Denote $z=r e^{i\theta}$ and $t=\log (-\log |z|^2)$.  Define 
 \begin{equation*}
\begin{split}
\chi_0 &= b dt \wedge 2 e^{-t} d\theta + p^* \chi_D \\
\omega_1 & = a dt \wedge 2 e^{-t} d\theta  + p^* \omega_D.
\end{split}
\end{equation*}
as metrics on $\Delta^* \times D$. Unlike the Poincar\'e type metric on a closed manifold, this metric is singular near $\partial \Delta$ which is an important property that we will use in this subsection.
In this subsection, we denote $\omega_{\varphi}= \omega_1 + dd^c \varphi$. We say that a function $\varphi \in C^{\infty}(\Delta^* \times D)$, if $\varphi$ and the covariant derivatives of $\varphi$ of any order defined using $\omega_1$ are bounded with respect to $\omega_1$.
Then we want to prove the following:
\begin{prop}\label{splitting}
Suppose that $\varphi \in C^{\infty}(\Delta^* \times D)$ is a $S^1$ invariant solution to 
\begin{equation*}
\omega_{\varphi}^n = \omega_{\varphi}^{n-1} \wedge \chi_0
\end{equation*} 
on $\Delta^* \times D$. Then we have that 
$\varphi$ is invariant in $\Delta^*$ direction.
\end{prop}
First, we prove the following Lemma:
\begin{lem}\label{dotvarphi laplacian}
For $\varphi$ as in Proposition \ref{splitting}, set $\dot{\varphi}= \partial_t \varphi$. Then we have that
\begin{equation}\label{linearized J}
(n \omega_{\varphi}^{n-1}-(n-1)\omega_{\varphi}^{n-2}\wedge \chi_0) \wedge dd^c \dot{\varphi}=0.
\end{equation}
\end{lem}
\begin{proof}
We use the complex coordinate $z= exp(-\frac{1}{2}e^t -i \theta)$ on $\Delta^*$ and denote the holomorphic vector field
\begin{equation*}
Z \triangleq Re [z \log z \frac{\partial}{ \partial z}].
\end{equation*}
Thus $Z$ can be written as 
\begin{equation*}
Z= \frac{1}{2} \frac{\partial}{\partial t} +\frac{1}{2} \theta \frac{\partial}{\partial \theta}, \text{ up to }\pi \frac{\partial}{\partial \theta}.
\end{equation*}
For any $S^1$ invariant function $f$, we have that
\begin{equation*}
Z \cdot f = \frac{1}{2} \dot{f}.
\end{equation*}
Note that $\mathcal{L}_Z \omega_1=\mathcal{L}_Z \omega_D + dd^c\Big(Z \cdot (-at) \Big)= dd^c (\frac{-a}{2} \dot{t})=0$ and similarly  $\mathcal{L}_Z \chi_0=0$.
Take the Lie derivative of the J-equation with respect to $Z$. We can get that
\begin{equation*}
\Big( n \omega_{\varphi}^{n-1} - (n-1) \omega_{\varphi}^{n-2} \wedge \chi_0 \Big) \wedge  dd^c \dot{\varphi}=0.
 \end{equation*}
\end{proof}
Note that $\omega_{\varphi}$ satisfies the J-equation 
\begin{equation*}
\omega_{\varphi}^n = \omega_{\varphi}^{n-1} \wedge \chi.
\end{equation*}
Choose a normal coordinate for $\chi$ such that in this coordinate $\omega_{\varphi}$ is diagonalized with diagonal elements to be $(\omega_{\varphi})_{i \bar i}= \lambda_i$. Then the J-equation becomes
\begin{equation*}
n= \Sigma_{i=1}^n \frac{1}{\lambda_i}.
\end{equation*}
Since by assumption we have that $\varphi \in C^{\infty}(X\setminus D)$. This implies that $\lambda_i \le C$ for some uniform constant $C$. Thus we have that
\begin{equation*}
n= \Sigma_{i=1}^n \frac{1}{\lambda_i} \ge \Sigma_{i\neq k} \frac{1}{\lambda_k} +\frac{1}{C}.
\end{equation*}
This implies that
\begin{equation*}
n \omega_{\varphi}^{n-1}- (n-1) \omega_{\varphi}^{n-2} \wedge \chi_0 \ge \frac{1}{C} \omega_{\varphi}^{n-1}.
\end{equation*}
This implies that (\ref{linearized J}) is a uniformly elliptic equation. However, Since $\omega_0$ doesn't have finite volume on $\Delta^* \times D$, we can not directly use integration by parts to get that $\dot{\varphi}$ is a constant.  As a result, first we prove the following Lemma:
\begin{lem}\label{lem l2 finite}
Let $\varphi$ be as in Proposition \ref{splitting}. Then we have that 
\begin{equation*}
\int_{\Delta^* \times D} e^t d \dot{\varphi} \wedge d^c \dot{\varphi} \wedge (n \omega_{\varphi}^{n-1}- (n-1) \omega_{\varphi}^{n-2} \wedge \chi_0) < +\infty.
\end{equation*}
\end{lem}
\begin{proof}
Denote $\Delta_s = \{z\in \Delta^*: -s \le t(z)\le s\}$.
Using integration by parts, we get that 
\begin{equation}\label{e l2 gradient}
\begin{split}
& \int_{\Delta_s \times D} e^t d \dot{\varphi} \wedge d^c \dot{\varphi} \wedge (n \omega_{\varphi}^{n-1}- (n-1) \omega_{\varphi}^{n-2} \wedge \chi_0) \\
& = - \int_{\Delta_s \times D} \dot{\varphi}  d e^t \wedge d^c \dot{\varphi} \wedge (n \omega_{\varphi}^{n-1}- (n-1) \omega_{\varphi}^{n-2} \wedge \chi_0) - \int_{\Delta_s \times D} e^t \dot{\varphi} dd^c \dot{\varphi} \wedge (n \omega_{\varphi}^{n-1}- (n-1) \omega_{\varphi}^{n-2} \wedge \chi_0) \\
& + \int_{\{t=s\}\times D} e^t \dot{\varphi}  d^c \dot{\varphi} \wedge (n \omega_{\varphi}^{n-1}- (n-1) \omega_{\varphi}^{n-2} \wedge \chi_0) - \int_{\{t=-s\}\times D}  e^t \dot{\varphi}  d^c \dot{\varphi} \wedge (n \omega_{\varphi}^{n-1}- (n-1) \omega_{\varphi}^{n-2} \wedge \chi_0)\\
&= -\int_{\Delta_s \times D} d e^t \wedge d^c \frac{1}{2}\dot{\varphi}^2 \wedge (n \omega_{\varphi}^{n-1}- (n-1) \omega_{\varphi}^{n-2} \wedge \chi_0) - 0 \\
& + \int_{\{t=s\}\times D} e^t \dot{\varphi}  d^c \dot{\varphi} \wedge (n \omega_{\varphi}^{n-1}- (n-1) \omega_{\varphi}^{n-2} \wedge \chi_0) - \int_{\{t=-s\}\times D}  e^t \dot{\varphi}  d^c \dot{\varphi} \wedge (n \omega_{\varphi}^{n-1}- (n-1) \omega_{\varphi}^{n-2} \wedge \chi_0)\\
&= - \int_{\Delta_s \times D}  \frac{\dot{\varphi}^2}{2} d^c d e^t \wedge (n \omega_{\varphi}^{n-1}- (n-1) \omega_{\varphi}^{n-2} \wedge \chi_0) \\
&+ \int_{\{t=s\}\times D} \big(e^t \dot{\varphi}  d^c \dot{\varphi} \wedge (n \omega_{\varphi}^{n-1}- (n-1) \omega_{\varphi}^{n-2} \wedge \chi_0) -\frac{1}{2}\dot{\varphi}^2 \wedge d^c e^t \wedge (n\omega_{\varphi}^{n-1}-(n-1)\omega_{\varphi}^{n-2}\wedge \chi_0) \big)\\
&- \int_{\{t=-s\}\times D}  \big(e^t \dot{\varphi}  d^c \dot{\varphi} \wedge (n \omega_{\varphi}^{n-1}- (n-1) \omega_{\varphi}^{n-2} \wedge \chi_0) -\frac{1}{2}\dot{\varphi}^2 \wedge d^c e^t \wedge (n\omega_{\varphi}^{n-1}-(n-1)\omega_{\varphi}^{n-2}\wedge \chi_0) \big) \\
& = \int_{\{t=s\}\times D} \big(e^t \dot{\varphi}  d^c \dot{\varphi} \wedge (n \omega_{\varphi}^{n-1}- (n-1) \omega_{\varphi}^{n-2} \wedge \chi_0) -\frac{1}{2}\dot{\varphi}^2 \wedge d^c e^t \wedge (n\omega_{\varphi}^{n-1}-(n-1)\omega_{\varphi}^{n-2}\wedge \chi_0) \big)\\
&- \int_{\{t=-s\}\times D}  \big(e^t \dot{\varphi}  d^c \dot{\varphi} \wedge (n \omega_{\varphi}^{n-1}- (n-1) \omega_{\varphi}^{n-2} \wedge \chi_0) -\frac{1}{2}\dot{\varphi}^2 \wedge d^c e^t \wedge (n\omega_{\varphi}^{n-1}-(n-1)\omega_{\varphi}^{n-2}\wedge \chi_0) \big)
\end{split}
\end{equation}
In the last line above, we use the fact that $dd^c e^t=0$. In fact, $d^c$ is defined by $d^c f = -\frac{1}{2}J df= -\frac{\sqrt{1}}{2}(\partial -\bar \partial) f$ and 
\begin{equation*}
dd^c e^t= -\frac{1}{2} dJd e^t = -\frac{1}{2} d (J e^t dt)= -\frac{1}{2} d (e^t \cdot 2e^{-t} d\theta)= 0    
\end{equation*}
Now it suffices to prove that all the boundary terms are uniformly bounded independent of $s$. In fact, since $\varphi \in C^{\infty}(\Delta^* \times D)$ and $\omega_{\varphi}$ is a Poincar\'e type metric, we have that 
\begin{equation*}
|  \dot{\varphi}  d^c \dot{\varphi} \wedge (n \omega_{\varphi}^{n-1}- (n-1) \omega_{\varphi}^{n-2} \wedge \chi_0) -\frac{1}{2}\dot{\varphi}^2 \wedge d^c t \wedge (n\omega_{\varphi}^{n-1}-(n-1)\omega_{\varphi}^{n-2}\wedge \chi_0) |\le C e^{-t} dvol_D \wedge \widetilde{\eta}.
\end{equation*}
Since $e^t$ cancels with $e^{-t}$, we get that 
\begin{equation*}
    \int_{\{t=s\}\times D} \big(e^t \dot{\varphi}  d^c \dot{\varphi} \wedge (n \omega_{\varphi}^{n-1}- (n-1) \omega_{\varphi}^{n-2} \wedge \chi_0) -\frac{1}{2}\dot{\varphi}^2 \wedge d^c e^t \wedge (n\omega_{\varphi}^{n-1}-(n-1)\omega_{\varphi}^{n-2}\wedge \chi_0) \big)
\end{equation*}
is uniformly bounded independent of $s$.
\end{proof}

Next, we prove the following Lemma:
\begin{lem}\label{zero l2}
Let $\varphi$ be as in Proposition \ref{splitting}. Then we have that 
\begin{equation*}
\int_{\Delta^* \times D} e^t d \dot{\varphi} \wedge d^c \dot{\varphi} \wedge (n \omega_{\varphi}^{n-1}- (n-1) \omega_{\varphi}^{n-2} \wedge \chi_0) =0.
\end{equation*}
\end{lem}
\begin{proof}
Denote 
\begin{equation*}
  \psi_{\varphi}(s)=   \int_{\{t=s\}\times D} \big(e^t \dot{\varphi}  d^c \dot{\varphi} \wedge (n \omega_{\varphi}^{n-1}- (n-1) \omega_{\varphi}^{n-2} \wedge \chi_0) -\frac{1}{2}\dot{\varphi}^2 \wedge d^c e^t \wedge (n\omega_{\varphi}^{n-1}-(n-1)\omega_{\varphi}^{n-2}\wedge \chi_0) \big).
\end{equation*}
According to (\ref{e l2 gradient}), it suffices to prove that $\psi_{\varphi}(s)$ is a constant function. Since $e^t d \dot{\varphi} \wedge d^c \dot{\varphi} \wedge (n \omega_{\varphi}^{n-1}- (n-1) \omega_{\varphi}^{n-2} \wedge \chi_0)\ge 0$, we can see from (\ref{e l2 gradient}) that $\psi_{\varphi}(s)$ is a nondecreasing function. According to  (\ref{e l2 gradient}), $\psi_{\varphi}$ is a bounded function. Thus $\psi_{\varphi}(\infty)= \lim_{s \rightarrow \infty} \psi_{\varphi}(s)$ and $\psi_{\varphi}(-\infty) = \lim_{s\rightarrow -\infty}\psi_{\varphi}(s)$ both exist and are finite. Denote $\Delta_{a,b}\triangleq \{b \le t \le a\}$ and $\Delta_{a}\triangleq \Delta_{-a,a}$. Consider an increasing sequence $(t_j)_{j \ge 0}$ going to $+\infty$ with $(t_{j+1}-t_j)_{j \ge 0}$ increasing to $\infty$. We can take $t_j = j^2$ to satisfy the above assumptions. Take $\alpha_j = \frac{t_{j+1}-t_j}{2}$. Denote $\varphi_j $ as $\varphi (\cdot +t_{j+1}-\alpha_j, \cdot)$. Denote $Z=2Re [z \log z \frac{\partial}{ \partial z}].$ Denote the flow induced by $Z$ as $\Phi^Z_t$. According to the proof of the Lemma \ref{dotvarphi laplacian}, we have that
\begin{equation*}
(\Phi_t^Z)^* \omega_0 = \omega_0,\,\,\,\, (\Phi_t^Z)^* \chi_0 = \chi_0
\end{equation*}
and
\begin{equation*}
(\Phi_t^Z)^* \omega_{\varphi}= \omega_0 + dd^c \varphi (\cdot + t, \cdot). 
\end{equation*}
Combining the above formulae and the fact that $\omega_{\varphi}$ is quasi-isometric to $\omega_0$, we can get that $\omega_{\varphi_j}$ are uniformly quasi-isometric to $\omega_0$ and
\begin{equation*}
tr_{\omega_{\varphi_j}}\chi_0 = (\Phi_{t_{j+1}-\alpha_j}^Z)^* tr_{\omega_{\varphi}} \chi_0= n.
\end{equation*}
Up to taking a subsequence, we have that $\{\varphi_j\}$ locally uniformly converge to a function $\varphi_{\infty}$ and 
\begin{equation*}
tr_{\omega_{\varphi_{\infty}}} \chi_0=n.
\end{equation*}
Using the Lemma \ref{lem l2 finite}, we have that 
\begin{equation*}
\int e^t d \dot{\varphi_{\infty}} \wedge d^c \dot{\varphi_{\infty}} \wedge (n \omega_{\varphi_{\infty}}^{n-1}- (n-1)\omega_{\varphi_{\infty}}^{n-2} \wedge \chi_0)<\infty.
\end{equation*}
Then we can calculate that
\begin{equation}\label{e varphi infty}
\begin{split}
&\int_{\Delta_{\alpha_j} \times D} e^t d \dot{\varphi_{\infty}} \wedge d^c \dot{\varphi_{\infty}} \wedge (n \omega_{\varphi_{\infty}}^{n-1}- (n-1)\omega_{\varphi_{\infty}}^{n-2} \wedge \chi_0) \\
&= \lim_{k \rightarrow \infty} \int_{\Delta_{\alpha_j} \times D} e^t d \dot{\varphi_{k}} \wedge d^c \dot{\varphi_{k}} \wedge (n \omega_{\varphi_{k}}^{n-1}- (n-1)\omega_{\varphi_{k}}^{n-2} \wedge \chi_0)\\
&  \le \lim_{k \rightarrow \infty} \int_{\Delta_{\alpha_k} \times D} e^t d \dot{\varphi_{k}} \wedge d^c \dot{\varphi_{k}} \wedge (n \omega_{\varphi_{k}}^{n-1}- (n-1)\omega_{\varphi_{k}}^{n-2} \wedge \chi_0) \\
& = \lim_{k \rightarrow \infty} \psi_{\varphi_k}(\alpha_k)- \psi_{\varphi_k}(-\alpha_k).
\end{split}
\end{equation}

We can compute that
\begin{equation*}
    d \dot{\varphi_k} = \ddot{\varphi_k} dt + d_D \dot{\varphi_k},
\end{equation*}
and
\begin{equation*}
    d^c \dot{\varphi_k}= -\frac{1}{2}J d \dot{\varphi_k} =-\frac{1}{2}J \ddot{\varphi_k} dt + d_D^c \dot{\varphi_k} =-e^{-t}\ddot{\varphi_k} d\theta + d_D^c \dot{\varphi_k},
\end{equation*}
and
\begin{equation*}
    dd^c \varphi_k = -e^{-t}\ddot{\varphi_k} dt \wedge d\theta + e^{-t}\dot{\varphi_k} dt \wedge d\theta -e^{-t}d_D \dot{\varphi_k} \wedge d\theta + dt \wedge d^c_D \dot{\varphi_k} + dd^c_D \varphi_k.
\end{equation*}
Thus we have that 
\begin{equation*}
    \omega_{\varphi_k}= \omega_D + dd^c_D \varphi_k + (a -\frac{1}{2}(\ddot{\varphi_k}-\dot{\varphi_k}))dt \wedge 2e^{-t}d\theta-e^{-t}d_D \dot{\varphi_k} \wedge d\theta + dt \wedge d^c_D \dot{\varphi_k} .
\end{equation*}
We compute that 
\begin{equation}\label{bd1}
\begin{split}
    &e^t d \dot{\varphi_k} \wedge d^c \dot{\varphi_k} \wedge n \omega_{\varphi_k}^{n-1}\\
    & = n (-\ddot{\varphi_k}^2 dt \wedge d\theta \wedge (\omega_D+ dd^c_D \varphi_k)^{n-1}+ \ddot{\varphi_k} dt \wedge d_D^c \dot{\varphi_k}\wedge (n-1)(-\frac{1}{2}d_D \dot{\varphi_k} \wedge 2 d\theta)\wedge (\omega_D+ dd^c_D \varphi_k)^{n-2} \\
    & + d_D \dot{\varphi_k} (-\ddot{\varphi_k})d\theta \wedge (n-1)dt \wedge d_D^c \dot{\varphi_k}\wedge (\omega_D + dd^c_D \varphi_k)^{n-2}  \\
    &+ d_D \dot{\varphi_k}\wedge d_D^c \dot{\varphi_k}\wedge (n-1)(a-\frac{1}{2}(\ddot{\varphi_k}-\dot{\varphi_k}))dt \wedge 2 d\theta \wedge (\omega_D + dd^c_D \varphi_k)^{n-2}).
\end{split}
\end{equation}
and
\begin{equation}\label{bd2}
\begin{split}
    &e^t d \dot{\varphi_k} \wedge d^c \dot{\varphi_k} \wedge  \omega_{\varphi_k}^{n-2}\wedge \chi_o\\
    & = n (-\ddot{\varphi_k}^2 dt \wedge d\theta \wedge (\omega_D+ dd^c_D \varphi_k)^{n-2}\wedge \chi_D  \\
    &+ \ddot{\varphi_k} dt \wedge d_D^c \dot{\varphi_k}\wedge (n-2)(-\frac{1}{2}d_D \dot{\varphi_k} \wedge 2 d\theta)\wedge (\omega_D+ dd^c_D \varphi_k)^{n-3}\wedge \chi_D \\
    & + d_D \dot{\varphi_k} (-\ddot{\varphi_k})d\theta \wedge (n-2)dt \wedge d_D^c \dot{\varphi_k}\wedge (\omega_D + dd^c_D \varphi_k)^{n-2}\wedge \chi_D  \\
    &+ d_D \dot{\varphi_k}\wedge d_D^c \dot{\varphi_k}\wedge (n-2)(a-\frac{1}{2}(\ddot{\varphi_k}-\dot{\varphi_k}))dt \wedge 2 d\theta \wedge (\omega_D + dd^c_D \varphi_k)^{n-3}\wedge \chi_D \\
    & + d_D \dot{\varphi_k}\wedge d^c_D \dot{\varphi_k}\wedge (\omega_D + dd^c \varphi_k)^{n-2}\wedge b dt \wedge 2d\theta).
\end{split}
\end{equation}

(\ref{bd1}) and (\ref{bd2}) imply that 
\begin{equation*}
\tau_{t_{k+1}-\alpha_k}^* e^t d \dot{\varphi_{k}} \wedge d^c \dot{\varphi_{k}} \wedge (n \omega_{\varphi_{k}}^{n-1}- (n-1)\omega_{\varphi_{k}}^{n-2} \wedge \chi_0) = e^t d \dot{\varphi} \wedge d^c \dot{\varphi} \wedge (n \omega_{\varphi}^{n-1}- (n-1)\omega_{\varphi}^{n-2} \wedge \chi_0),
\end{equation*}
where $\tau_{t_{k+1}-\alpha_k}$ is a translation in the $t$ direction. Thus we have that:
\begin{equation}\label{e phi tran}
    \psi_{\varphi_k}(\alpha_k)- \psi_{\varphi_k}(-\alpha_k) = \psi_{\varphi}(t_{k+1}) - \psi_{\varphi}(t_k).
\end{equation}
Combining (\ref{e varphi infty}), (\ref{e phi tran}) and the fact that $\lim_{k \rightarrow \infty} \psi_{\varphi}(t_{k+1}) - \psi_{\varphi}(t_k)=0$, we can get that 
\begin{equation*}
\int_{\Delta^* \times D}e^t d \dot{\varphi_{\infty}} \wedge d^c \dot{\varphi_{\infty}} \wedge (n \omega_{\varphi_{\infty}}^{n-1}- (n-1)\omega_{\varphi_{\infty}}^{n-2} \wedge \chi_0) =0.
\end{equation*}
Thus, $\dot{\varphi_{\infty}}$ is a constant function. Since $\varphi_{\infty}$ is bounded, we must have that $\dot{\varphi_{\infty}}=0$.  
Thus we can use the dominated convergence theorem to get that
\begin{equation*}
\lim_{j \rightarrow \infty} \psi_{\varphi}(\frac{t_{j+1}+t_j}{2}) = \lim_{j \rightarrow \infty} \psi_{\varphi_j}(0)= \psi_{\varphi_{\infty}}(0).
\end{equation*}
Since $\dot{\varphi_{\infty}}=0$, we have that
\begin{equation*}
\psi_{\varphi_{\infty}} (0)= \int_{\{t=0\}\times D} \dot{\varphi_{\infty}}  d^c \dot{\varphi_{\infty}} \wedge (n \omega_{\varphi_{\infty}}^{n-1}- (n-1) \omega_{\varphi_{\infty}}^{n-2} \wedge \chi_0)=0.
\end{equation*}
Thus we have that $\psi_{\varphi}(\infty)=0$. Similarly we can also get that $\psi_{\varphi}(-\infty)=0$. This concludes the proof of this Lemma.
\end{proof}

Now we are ready to proof the main proposition in this subsection:
\begin{proof}
(of the Proposition \ref{splitting}). Using the Lemma \ref{lem l2 finite}, we can get that
\begin{equation*}
\dot{\varphi}=C
\end{equation*}
for some constant $C$. Since $\varphi$ is bounded on $\Delta^* \times D$, we must have that $\dot{\varphi}=0$. This concludes the proof of this Proposition.
\end{proof}

\subsection{Almost solving J-equation}
We define
\begin{defn}
We say that a family of K\"ahler metrics $\omega_s$ on a manifold $D$ almost solve the J-equation if the following hold:
\begin{enumerate}
\item For any $k \ge 0$, $tr_{\omega_s}\chi \rightarrow n$ and for any positive $l$, $\partial_s^l tr_{\omega_s}\chi  \rightarrow 0$, in $C^k(D)$, as $s$ goes to $\infty$.
\item $(\omega_s)_{s \ge 0}$ is bounded in $C^k$ for any $k$, and there is some positive constant $c$ such that for all $t \ge 0$, $\omega_s \ge c \chi$.
\end{enumerate}
We say moreover that such a family has extinguishing variation if for all positive $l$, $\partial^l_t \omega_t$ tends to $0$ in every $C^k(D)$, $k \ge 0$, as $t$ goes to $\infty$.
\end{defn}
\begin{rem}\label{rem almost solve J}
Using maximum principle, we can get that the solution to a J-equation is unique, denoted as $\omega_{\infty}$. Thus for any family of K\"ahler metrics $\omega_s$ on a manifold $D$ which almost solve the J-equation, we have that $\omega_s \rightarrow \omega_{\infty}$ in $C^k$ for any $k$.
\end{rem}

\begin{prop}\label{p almost J}
Assume  that $D$ is reduced to one component and that $tr_{\omega_{\varphi}}\chi=n$ on $X\setminus D$, with $\omega_{\varphi}= \omega+ dd^c \varphi$ of Poincar\'e type of class $[\omega_X]$. Let $\varphi= \varphi_0 + \varphi^{\bot}$ be the decomposition of $\varphi$ as in the section 2.5. Then for $T$ big enough, $(\omega_t^{\varphi})_{t\ge T}\triangleq (\omega_D + dd^c_D \varphi_0 (t,\cdot))_{t\ge T}$, with $\omega_D = \omega_X |_D$, is a family of K\"ahler metrics on $D$ which almost solve the J-equation, with extinguishing variation. Moreover, for any $(k,\alpha)$, and any positive $l$, $(\partial_t^l \varphi_0)_t \rightarrow 0$ in $C^{k,\alpha}(D)$ as $t$ goes to $\infty$.
\end{prop}
 \begin{proof}
According to the Lemma \ref{improve decay holder}, $dd^c \varphi^{\bot}$ decays exponentially with respect to a Poincar\'e type metric. Thus, we have that $\omega_t^{\varphi}$ are K\"ahler metrics which are uniformly bounded from below and above for $t \ge A$ with $A$ big enough.  Then, we prove by contradiction. Assume that there exists $\epsilon>0$, a sequence $(t_j)$, $\lim_{j \rightarrow \infty} t_j = \infty$, and $(z_j)$ in $D$ such that $|tr_{\omega_{t_j}^{\varphi}}\chi_D-C_D|(z_j)\ge \epsilon$. We consider a subsequence of $(z_j)$, still denoted as $(z_j)$, converging to some $z\in D$. Note that for any $(k,\alpha) \in \mathbb{N} \times (0,1)$, $\varphi_0$ is bounded in $C^{k,\alpha}([A,\infty)\times D)$, according to the Proposition \ref{prop bdd}.  As a result, we can use a diagonal argument to get a function $\varphi_{\infty} \in C^{\infty}(\mathbb{R}\times D)$ such that $(\varphi_0)_j \triangleq (\varphi_0)(\cdot + t_j,\cdot)$ converges in $C^{\infty}([-N, N]\times D)$ for any $N$. We claim that: \\

The $(1,1)-$form $\omega_{\varphi_{\infty}}\triangleq \omega_D + dd^c (\varphi_{\infty}-at)$ solves the J-equation on $\Delta^* \times D$. 

In fact, near the divisor, we have that
\begin{equation}\label{expansion 1}
\omega_{\varphi}= (a-\frac{1}{2} (\partial_t^2 -\partial_t)\varphi_0) dt \wedge 2 e^{-t}\eta + dt \wedge d^c_D \partial_t (\varphi_0) -\frac{1}{2} d_D \partial_t (\varphi_0) \wedge 2e^{-t}\eta + p^* \omega_t^{\varphi} + O(e^{-t}).
\end{equation}

We calculate that
\begin{equation*}
\begin{split}
&n= tr_{\omega_{\varphi}}\chi = n \frac{\omega_{\varphi}^{n-1} \wedge \chi}{\omega_{\varphi}^n} \\
&= n \frac{b  (\omega_t^{\varphi})^{n-1} +(n-1)p^* \chi_D \wedge \Big(a-\frac{1}{2} (\partial_t^2 -\partial_t) \varphi_0 \Big) \wedge (\omega_t^{\varphi})^{n-2}  }{n \Big(a-\frac{1}{2} (\partial_t^2 -\partial_t)\varphi_0 \Big) (\omega_t^{\varphi})^{n-1}+\frac{n(n-1)}{2} d_D \partial_t(\varphi_0) \wedge d_D^c \partial_t(\varphi_0) \wedge (\omega_t^{\varphi})^{n-2}}\\
& +n \frac{ \frac{(n-1)(n-2)}{2}p^* \chi_D \wedge d_D \dot{\varphi_0} \wedge d_D^c \dot{\varphi_0}  \wedge (\omega_{t}^{\varphi})^{n-3}}{n \Big(a -\frac{1}{2} (\partial_t^2 -\partial_t)\varphi_0 \Big) (\omega_t^{\varphi})^{n-1}+\frac{n(n-1)}{2} d_D \partial_t(\varphi_0) \wedge d_D^c \partial_t(\varphi_0) \wedge (\omega_t^{\varphi})^{n-2}}\\
& +O(e^{-\eta t}). 
\end{split}
\end{equation*}
and
\begin{equation*}
\begin{split}
&n= tr_{\omega_{\varphi_{\infty}}}\chi = n \frac{\omega_{\varphi_{\infty}}^{n-1} \wedge \chi}{\omega_{\varphi_{\infty}}^n} \\
&= n \frac{b  (\omega_t^{\varphi_{\infty}})^{n-1} +(n-1)p^* \chi_D \wedge \Big(a-\frac{1}{2} (\partial_t^2 -\partial_t)  \varphi_{\infty} \Big) \wedge (\omega_t^{\varphi_{\infty}})^{n-2}  }{n \Big(a-\frac{1}{2} (\partial_t^2 -\partial_t) \varphi_{\infty} \Big) (\omega_t^{\varphi_{\infty}})^{n-1}+\frac{n(n-1)}{2} d_D \partial_t(\varphi_{\infty}) \wedge d_D^c \partial_t( \varphi_{\infty}) \wedge (\omega_t^{\varphi_{\infty}})^{n-2}}\\
& +n \frac{ \frac{(n-1)(n-2)}{2}p^* \chi_D \wedge d_D^c \dot{ \varphi_{\infty}} \wedge d_D \dot{\varphi_{\infty}}  \wedge (\omega_{t}^{\varphi_{\infty}})^{n-3}}{n \Big(a-\frac{1}{2} (\partial_t^2 -\partial_t) \varphi_{\infty} \Big) (\omega_t^{\varphi_{\infty}})^{n-1}+\frac{n(n-1)}{2} d_D \partial_t( \varphi_{\infty}) \wedge d_D^c \partial_t(\varphi_{\infty}) \wedge (\omega_t^{\varphi_{\infty}})^{n-2}}\\
& +O(e^{-\eta t}). 
\end{split}
\end{equation*}

Since we have that $(\varphi_0)_j$ locally compactly converges to $\varphi_{\infty}$, we can use the above formulae to get that $\varphi_{\infty}$ solves the J-equation. Then we can use the Proposition \ref{splitting} to get that $\varphi_{\infty}$ is invariant in the $\Delta^*$ direction. Thus, we can write
\begin{equation*}
\omega_{\varphi_{\infty}}= a dt \wedge 2 e^{-t} d\nu + \omega_D^{\psi},
\end{equation*}
with $\omega_D^{\psi}= \omega_D + dd^c_D \psi$, for $\psi \in C^{\infty}(D)$. Since $\omega_D^{\psi}$ is the $C^2-$limit of $(\omega_{t_j}^{\varphi})$, the fact that $\omega_{\varphi_{\infty}}$ solves a J-equation contradicts with the assumption that $|tr_{\omega_{t_j}^{\varphi}}\chi_D-C_D|\ge \epsilon$. This concludes the proof of the first part of the Proposition. The second part of the Proposition can be proved in a similar way.
 \end{proof}
 
\begin{prop}\label{asymp 1}
 Let $\omega_{\varphi}$ be a Poincar\'e type K\"ahler metric satisfying the J-equation
 \begin{equation*}
 tr_{\omega_{\varphi}}\chi = n.
 \end{equation*}
 Then we have that
 \begin{equation*}
 \omega_{\varphi}= p^* \omega_D + a dt \wedge 2e^{-t} \eta +o(1).
 \end{equation*}
\end{prop}
\begin{proof}
According to Proposition \ref{p almost J}, we have that $\partial_t \varphi_0 =o(1)$ and $\partial_t^2 \varphi_0 =o(1)$ at any order with respect to a Poincar\'e type K\"ahler metric. Thus we can use (\ref{expansion 1}) to get that
\begin{equation}\label{omega varphi 1}
\omega_{\varphi} = a  dt \wedge 2 e^{-t}\eta + p^* \omega_t^{\varphi} + o(1).
\end{equation}
According to the Proposition \ref{p almost J}, we have that $(\omega_t^{\varphi})_{t\ge T}$ is a family of K\"ahler metrics which almost solve the J-equation, with extinguishing variation. According to the remark \ref{rem almost solve J}, we have that $\omega_t^{\varphi}= \omega_D +o(1)$. Plug this into (\ref{expansion 1}), we get that
\begin{equation*}
\omega_{\varphi} = a  dt \wedge 2 e^{-t}\eta + p^* \omega_D + o(1).
\end{equation*}
\end{proof}

Then we can prove the following Proposition:
\begin{prop}
Let $\widetilde{\Delta}$ be the linearized operator of the J-equation defined by
\begin{equation*}
\widetilde{\Delta} u = g_{\varphi}^{k \bar j} g_{\varphi}^{i \bar l} \chi_{i \bar j} u_{k \bar l}.
\end{equation*} 
For any $k$, one has on $\widetilde{\Delta}^0 \triangleq \Pi_0 \circ \widetilde{\Delta} \circ q^* : C^{k+2,\alpha} ([A,\infty)\times \mathbb{R}) \rightarrow C^{k,\alpha}([A,\infty)\times \mathbb{R})$ the asymptotic:
\begin{equation*}
\widetilde{\Delta}u=<\chi_D, dd^c_D u>_{\omega_D} -\frac{b}{a^2} (\partial_t - \partial_t^2)u +o(1).
\end{equation*}
\end{prop}
\begin{proof}
This proposition follows immediately from the Proposition \ref{asymp 1} and the asymptotic behaviour of $\chi$.
\end{proof}

Next, we want to improve $o(1)$ in the above Proposition to $O(e^{-\eta t})$. 

Define $\widetilde{\Delta}^0: C^{k+2,\alpha}([A,\infty)\times D) \rightarrow C^{k,\alpha}([A,\infty)\times D)$ as
\begin{equation*}
\widetilde{\Delta}^0 u = <\chi_D, dd^c_D u> _{\omega_D} -\frac{b}{a^2} (\partial_t -\partial_t^2)u.
\end{equation*}

 For any function $u$ defined on $D \times [0,+\infty)$, we can decompose it into to parts $u=v+ w$, where $v$ is equal to some constant on each fiber $D\times \{t\}$ for each $t$ although such constant can be different for different $t$ and $w$ has zero integral in each fiber. 
 
 \begin{defn}
 Define the space 
 \begin{equation*}
 C^{k,\alpha}_{00,\eta} ([A,\infty)\times D)\triangleq \{u \in C^{k,\alpha}_{\eta} ([A,\infty)\times D): u|_{t=A}=0, \int_{\{t\}\times D} u dvol_D =0 \text{ for any } t\ge A \},
 \end{equation*}
  \begin{equation*}
 C^{k,\alpha}_{0,\eta} ([A,\infty)\times D)\triangleq \{u \in C^{k,\alpha}_{\eta} ([A,\infty)\times D): \int_{\{t\}\times D} u dvol_D =0 \text{ for any } t\ge A \},
 \end{equation*}
 \begin{equation*}
 C_{0,\eta}^{k,\alpha}([A,\infty))= \{u \in \mathbb{R} \oplus C_{\eta}^{k,\alpha}([A,\infty)) : u|_{t=A}=0\}.
 \end{equation*}
 \begin{equation*}
 W^{k,2}_{00,\eta} ([A,\infty)\times D)\triangleq \{u \in W^{k,2}_{\eta} ([A,\infty)\times D): u|_{t=A}=0, \int_{\{t\}\times D} u dvol_D =0 \text{ for any } t\ge A \},
 \end{equation*}
 and
  \begin{equation*}
 W^{k,2}_{0,\eta} ([A,\infty)\times D)\triangleq \{u \in W^{k,2}_{\eta} ([A,\infty)\times D):  \int_{\{t\}\times D} u dvol_D =0 \text{ for any } t\ge A \},
 \end{equation*}
 \end{defn}
 
 \begin{rem}\label{rem decom}
 Note that for $\eta <0$, $\{u \in \mathbb{R} \oplus C_{\eta}^{k,\alpha}([A,\infty) \times D): u|_{t=A}=0\}= C^{k,\alpha}_{00,\eta} ([A,\infty)\times D) \oplus C_{0,\eta}^{k,\alpha}([A,\infty))$. In fact, for $u=a+u^* \in \{u \in \mathbb{R} \oplus C_{\eta}^{k,\alpha}([A,\infty) \times D): u|_{t=A}=0\}$, we can consider $u_0\in C_{0,\eta}^{k,\alpha}([A,\infty))$ defined by $u_0(t) \triangleq \int_{D\times \{t\}} u \omega_D^n$. Since $\lim_{t \rightarrow \infty} u^* =0$, we have that
 \begin{equation*}
 \lim_{t \rightarrow \infty} u= \lim_{t \rightarrow \infty} u_0  =a .
 \end{equation*}
 Thus we have that $u-u_0 \in C_{\eta}^{k,\alpha}([A,\infty)\times D)$ which implies that $u-u_0 \in C^{k,\alpha}_{00,\eta} ([A,\infty)\times D)$.
 \end{rem}
 
 Before we proceed, we record the following Lemma:
 \begin{lem}\label{omegad j}
 Suppose that $\omega$ and $\chi$ are  Poincar\'e type K\"ahler metrics on $X$ which are locally asymptotic to a product metric near $D$, i.e.
 \begin{equation}\label{omega asy}
 \omega= p^* \omega_D - 2a e^{-t} dt \wedge d\theta + O(e^{-\eta t})
 \end{equation}
 and
 \begin{equation}\label{chi asy}
 \chi= p^* \chi_D - 2b e^{-t} dt \wedge d\theta + O(e^{-\eta t}).
 \end{equation}
 Suppose that $\omega$ satisfies that 
\begin{equation}\label{omega d}
 \frac{{\omega}^{n-1}\wedge \chi}{\omega^{n}}= \frac{1}{C} +O(e^{-\eta t}).
\end{equation}
Then we have that 
\begin{equation*}
tr_{\omega_D} \chi_D = C_D
\end{equation*}
for some constant $C_D$.
 \end{lem}
 \begin{proof}
 (\ref{omega asy}) and (\ref{chi asy}) imply that 
 \begin{equation*}
 \frac{\omega^{n-1}\wedge \chi}{\omega^n} = \frac{p^* \omega_D^{n-1} \wedge (-2b)e^{-t}dt \wedge d\theta + (n-1) p^* \omega_D^{n-2} \wedge p^* \chi_D \wedge (-2a) e^{-t}dt \wedge d\theta}{ n p^* \omega_D^{n-1} \wedge (-2a)e^{-t}dt \wedge d\theta } +O(e^{-\eta t}).
 \end{equation*}
 Then we restrict to $D$, we can use (\ref{omega d}) to get that
 \begin{equation*}
\frac{\omega_D^{n-2}\wedge \chi_D}{\omega_D^{n-1}} =\frac{n}{(n-1)C} -\frac{b}{a(n-1)}.
 \end{equation*}
 This concludes the proof of this Lemma.
 \end{proof}

 \begin{lem}\label{lem iso1}
 There exists $\delta_0$ small enough such that for any $0< \eta \le \delta_0$, we have that
 \begin{enumerate}
 \item $\widetilde{\Delta}^0:  C^{k+2,\alpha}_{00,-\eta} ([A,\infty)\times D) \rightarrow C^{k,\alpha}_{0,-\eta} ([A,\infty)\times D)$ is an isometry.
 \item $\widetilde{\Delta}^0:  W^{k+2,2}_{00,-\eta} ([A,\infty)\times D) \rightarrow  W^{k,2}_{0,-\eta} ([A,\infty)\times D)$ is an isometry.
 \end{enumerate}
 \end{lem}
 \begin{proof}
 We can calculate that
 \begin{equation*}
 \begin{split}
 &\int_{[A,\infty)\times D} w \widetilde{\Delta}^0 w e^{2\delta t} dt dVol_D \\
 &= \int_{[A,\infty)\times D} <\chi_D, dd^c_D w>_{\omega_D} w e^{2\delta t} dt dvol_D -\frac{b}{a^2} \int_{[A,\infty)\times D} (\partial_t- \partial_t^2)w \cdot w e^{2\delta t}dt dvol_D\\
 & = -  \int_{[A,\infty)\times D}  g_D^{\alpha \bar \beta} (\chi_D)_{\bar \beta l} g_D^{l \bar k} w_{\bar k} w_{\alpha} e^{2\delta t} dt dvol_D -\frac{b}{a^2} \Big(  \int_{[A,\infty)\times D}  e^{2\delta t} w_t^2 dt dvol_D -\delta (1+2\delta) \int e^{2\delta t}w^2 dt dvol_D \Big)\\
 &\le -C  \int_{[A,\infty)\times D} |\nabla_D w|^2 e^{2\delta t} dt dvol_D -(C-\frac{b\delta (1+2\delta)}{a^2})  \int_{[A,\infty)\times D} e^{2\delta t}w^2 dt dvol_D \\
 &-\frac{b}{a^2}  \int_{[A,\infty)\times D}  e^{2\delta t}w_t^2 dt dvol_D
 \end{split}
 \end{equation*}
In the third line above, we use the fact that $\omega_D$ solves the J-equation $tr_{\omega_D}\chi_D =C_D$ on $D$ according to the Lemma \ref{omegad j}, which enables us to do integration by parts for the first integral. In the fourth line we use the fact that $w$ has zero integral on each fiber $D\times \{t\}$ such that we can use Poincar\'e Inequality on each $D\times\{t\}$. For $\delta>0$ small enough, the above formula implies that
\begin{equation*}
\int_{[A,\infty)\times D} w \widetilde{\Delta}^0 w e^{2\delta t} dt dVol_D \le -C_0 \int_{[A,\infty)\times D} ( |\nabla_D w|^2 + w^2 + w_t^2)e^{2\delta t} dt dvol_D.
\end{equation*}
This concludes the proof of part (2) of this Proposition. Part (1) can be proved in a similar way.
\end{proof}

\begin{lem}\label{lem iso2}
For any $\eta<0$, we have that
\begin{equation*}
\widetilde{\Delta}^0 : C^{k+2,\alpha}_{0,\eta}([A,\infty)) \rightarrow C^{k,\alpha}_{\eta}([A,\infty))
\end{equation*}
is an isometry.
\end{lem}
\begin{proof}
Given $g_0 \in C^{k,\alpha}_{\eta}([A,\infty))$, we can define
\begin{equation*}
v(t)= \int_A^t e^s ds \int_s^{+\infty} e^{-u} \frac{g_0(u)}{a} du.
\end{equation*}
By direct calculation, we can get that
\begin{equation*}
\widetilde{\Delta}^0 v=-a (\partial_t^2 -\partial_t)v= g_0.
\end{equation*}
Note that $v(A)=0$, then we have that
\begin{equation*}
v(t) \in C^{k+2,\alpha}_{0,\eta}([A,\infty)).
\end{equation*}
This concludes the proof of this Lemma.
\end{proof}

Then we can prove that:
 \begin{prop}\label{solve widetilde delta}
 Suppose that $\omega$ satisfies J-equation. Then there exists $\delta_0>0$ such that for any $\eta \in (0, \delta_0)$,
 \begin{equation*}
 \widetilde{\Delta}:  \{u \in C^{k,\alpha}_{-\eta}([A,\infty)\times D) \oplus \mathbb{R}: u|_{t=A}=0\} \rightarrow C^{k,\alpha}_{-\eta}([A,\infty)\times D)
 \end{equation*}
 is an isometry.
 \end{prop}
\begin{proof}
From the Lemma \ref{lem iso1} and the Lemma \ref{lem iso2}, we have that $\widetilde{\Delta}:  C^{k+2,\alpha}_{00,-\eta} ([A,\infty) \times D) \oplus C^{k+2,\alpha}_{0,-\eta}([A,\infty)) \rightarrow C^{k,\alpha}_{0,-\eta}  ([A,\infty) \times D) \oplus C^{k,\alpha}_{-\eta}([A,\infty))$ is an isometry. 
According to the Remark \ref{rem decom}, we have
\begin{equation*}
C^{k+2,\alpha}_{00,-\eta}  ([A,\infty) \times D) \oplus C^{k+2,\alpha}_{0,-\eta}([A,\infty)) =  \{u \in C^{k,\alpha}_{-\eta}([A,\infty)\times D)\oplus \mathcal{R}: u|_{t=A}=0\}
\end{equation*}
and
\begin{equation*}
C^{k,\alpha}_{0,-\eta}  ([A,\infty) \times D) \oplus C^{k,\alpha}_{-\eta}([A,\infty)) = C^{k,\alpha}_{-\eta}([A,\infty)\times D).
\end{equation*}
This concludes the proof of this proposition.
\end{proof}

Now we are ready to prove the main theorem in this section:
\begin{proof}
(of the Theorem \ref{thm asy})
As before, we use the holomorphic vector field $Z= Re [z^n \log z^n \partial_{z^n}]$ to act on the J-equation
\begin{equation*}
tr_{\omega_{\varphi}}\chi=n.
\end{equation*}
Note that 
\begin{equation*}
\omega_{\varphi}= \omega+ dd^c \varphi= \omega_0  + dd^c \varphi_0 + O(e^{-\eta t})
\end{equation*}
and
\begin{equation*}
\chi= \chi_0 + O(e^{-\eta t}).
\end{equation*}
Thus, we have that
\begin{equation*}
\mathcal{L}_Z \omega_{\varphi}= \mathcal{L}_Z  \omega_0 + dd^c (Z\cdot  \varphi_0) +O(e^{-\eta t})= dd^c \frac{\partial_t (\varphi_0)}{2} +O(e^{-\eta t})
\end{equation*}
and
\begin{equation*}
\mathcal{L}_Z \chi = \mathcal{L}_Z \chi_0+  O(e^{-\eta t}) =O(e^{-\eta t}).
\end{equation*}
Then we have that 
\begin{equation*}
0= \mathcal{L}_Z n=\mathcal{L}_Z (tr_{\omega_{\varphi}}\chi)= -<\chi, \mathcal{L}_Z \omega_{\varphi}>_{\omega_{\varphi}} + tr_{\omega_{\varphi}} \mathcal{L}_Z \chi= -\frac{1}{2}<\chi, dd^c \partial_t(\varphi_0)>_{\omega_{\varphi}} + O(e^{-\eta t}).
\end{equation*}
Thus we have that $\widetilde{\Delta}(\partial_t(\varphi_0))= O(e^{-\eta t})$ which implies that 
\begin{equation}\label{delta varphit}
\widetilde{\Delta}^0(\chi \partial_t(\varphi_0))= O(e^{-\eta t}).
\end{equation}
Here $\chi$ is a cut-off function which is equal to $1$ for $t \ge 2A$ and vanish near $t=A$.
According to the Proposition \ref{solve widetilde delta},   $\widetilde{\Delta}^0$ is an isometry from $ \{u \in C^{k,\alpha}_{-\eta}([A,\infty)\times D) \oplus \mathbb{R}: u|_{t=A}=0\}$ to $C^{k,\alpha}_{-\eta}([A,\infty)\times D)$. Thus, (\ref{delta varphit}) implies that $\chi \partial_t(\varphi_0) \in \{u \in C^{k,\alpha}_{-\eta}([A,\infty)\times D) \oplus \mathbb{R}: u|_{t=A}=0\}$. Taking an integral with respect to $t$, we can get that
\begin{equation*}
\varphi_0 \in p^*C^{k+2,\alpha}(D) \oplus C_{-\eta}^{k+2,\alpha}(X\setminus D) \oplus t \mathbb{R}.
\end{equation*}
Thus we can write $\varphi_0 = p^*v_1 + v_2 +ct$, where $v_1 \in C^{k+2,\alpha}(D)$ and $v_2 \in  C_{-\eta}^{k+2,\alpha}$.
Using the Lemma \ref{improve decay holder} and the fact that $\varphi \in C^{\infty}$, we can get that $\varphi^{\bot}\in C_{-1}^\infty$ and $c=0$. Then we can get that
\begin{equation*}
\omega_{\varphi}= \omega_0 + O(e^{-\eta t }) + dd^c \varphi_0 + dd^c \varphi^{\bot} = \omega_0 + dd^c v_1 + dd^c v_2 = (\omega_D+ dd^c_D v_1) - a dt \wedge 2e^{-t  } \widetilde{\eta}  + O(e^{-\eta t}). 
\end{equation*}
This concludes the proof of the asymptotic behaviour of $\omega_{\varphi}.$
\end{proof}

\end{document}